\documentclass[12pt, reqno]{amsart}

\usepackage[a4paper, centering, total={170mm,240mm}]{geometry}

\usepackage{amssymb,latexsym}
\usepackage{blindtext}
\usepackage[backgroundcolor=olive, linecolor=olive, textsize=tiny]{todonotes}
\usepackage{esint,dsfont}

\usepackage{palatino}

\usepackage[colorlinks=true, linkcolor=blue, citecolor=purple, urlcolor=blue]{hyperref}

\def\R{{\mathbb R}}
\def\N{\mathbb{N}}
\def\C{\mathbb{C}}

\def\D{\mathbb{D}}

\def\P{\mathbb{P}}

\newtheorem{prop}{\bf Proposition}[section]
\newtheorem{thm}[prop]{\bf Theorem}
\newtheorem{cor}[prop]{\bf Corollary}
\newtheorem{lem}[prop]{\bf Lemma}
\newtheorem{rmk}[prop]{\it Remark}

\begin{document}

\title[Property (T) for uniformly bounded representations]{{\bf\Large Property (T) for uniformly bounded representations and weak*-continuity of invariant means}}

\author[I. Vergara]{Ignacio Vergara}
\address{Departamento de Matem\'atica y Ciencia de la Computaci\'on, Universidad de Santiago de Chile, Las Sophoras 173, Estaci\'on Central 9170020, Chile}

\email{ign.vergara.s@gmail.com}
\thanks{This work is supported by the FONDECYT project 3230024. Part of the research was carried out at the Leonhard Euler St. Petersburg International Mathematical Institute and supported by the Ministry of Science and Higher Education of the Russian Federation (Agreement № 075–15–2022–287 dated 06.04.2022).}

\makeatletter
\@namedef{subjclassname@2020}{%
  \textup{2020} Mathematics Subject Classification}
\makeatother

%


\begin{abstract}
For every $c\geq 1$, we define a strengthening of Kazhdan's Property (T) by considering uniformly bounded representations $\pi$ with fixed bound $|\pi|\leq c$. We carry out a systematic study of this property and show that it can be characterised by the weak*-continuity of the unique invariant mean on a suitable space of coefficients. For countable groups, we prove that the family of properties thus obtained yield an invariant at the von Neumann algebra level. Moreover, by focusing on certain representations of rank 1 Lie groups, we show that $\operatorname{Sp}(n,1)$ and $F_{4,-20}$ admit proper uniformly Lipschitz affine actions on Hilbert spaces.
\end{abstract}


\begingroup
\def\uppercasenonmath#1{} 
\let\MakeUppercase\relax 
\maketitle
\endgroup

\section{{\bf Introduction}}
Property (T) was introduced by Kazhdan \cite{Kaz} as a tool to prove that certain lattices in Lie groups are finitely generated. Ever since, it has become one of the most important properties in analytic group theory, with applications in several fields of mathematics. We refer the reader to \cite{BedlHVa} for a thorough account on this subject.

In \cite{HaKndL}, Haagerup, Knudby and de Laat proved that Property (T) for a locally compact group $G$ can be characterised by the weak*-continuity of the unique invariant mean on the Fourier--Stieltjes algebra $B(G)$. They also introduced a strengthening of this property, called Property (T${}^*$), by requiring this invariant mean to be weak*-continuous as an element of the dual space of the algebra of completely bounded Fourier multipliers $M_0A(G)$. In this paper, we study a family of properties that sit in between Property (T) and Property (T${}^*$), by looking at the spaces of coefficients of uniformly bounded representations $B_c(G)$ ($c\geq 1$).

Let $G$ be a locally compact group, $\mathcal{H}$ a Hilbert space, and $\mathcal{B}(\mathcal{H})$ the algebra of bounded operators on $\mathcal{H}$. A representation $\pi:G\to\mathcal{B}(\mathcal{H})$ is \textit{uniformly bounded} if
\begin{align*}
|\pi|=\sup_{s\in G}\|\pi(s)\| < \infty.
\end{align*}
We only consider representations that are continuous for the strong operator topology (SOT). For $c\geq 1$, let $\mathcal{R}_c(G)$ denote the class of all those representations $\pi$ such that $|\pi|\leq c$. We say that $G$ has Property (T)${}_c$ if, for every $\pi\in\mathcal{R}_c(G)$, the existence of almost invariant vectors implies the existence of non-trivial invariant vectors; see Section \ref{Sec_(T)_c} for a precise definition.

The space of coefficients of representations in $\mathcal{R}_c(G)$ is denoted by $B_c(G)$. This space consists of all the continuous functions $\varphi:G\to\C$ that can be written as
\begin{align}\label{phi_intro}
\varphi(s)=\langle\pi(s)\xi,\eta\rangle,\quad\forall s\in G,
\end{align}
where $\pi\in\mathcal{R}_c(G)$, and $\xi,\eta$ are elements of the associated Hilbert space. This is a Banach space for the norm
\begin{align*}
\|\varphi\|_{B_c(G)}=\inf \|\xi\|\|\eta\|,
\end{align*}
where the infimum is taken over all the decompositions as in \eqref{phi_intro}. Moreover, $B_c(G)$ is a subspace of the algebra of weakly almost periodic functions $\operatorname{WAP}(G)$, and it always carries a unique invariant mean; see Corollary \ref{Cor_inv_mean}. Furthermore, it has a natural predual. The algebra $\tilde{A}_c(G)$ is defined as the completion of $L^1(G)$ for the norm
\begin{align*}
\|f\|_{\tilde{A}_c(G)}=\sup_{\pi\in\mathcal{R}_c(G)}\|\pi(f)\|,
\end{align*}
where $\pi(f)=\int_G f(s)\pi(s)\,ds$. Then the dual space of $\tilde{A}_c(G)$ can be identified with $B_c(G)$; see Proposition \ref{Prop_duality}. We prove the following.

\begin{thm}\label{Thm_equiv}
Let $G$ be a locally compact group and $c\geq 1$. The following are equivalent:
\begin{itemize}
\item[(i)] The group $G$ has Property (T)${}_c$.
\item[(ii)] The unique invariant mean on $B_c(G)$ is $\sigma(B_c(G),\tilde{A}_c(G))$-continuous.
\end{itemize}
\end{thm}

Observe that $\mathcal{R}_1(G)$ is the class of unitary representations of $G$, hence Property (T)${}_1$ coincides with Kazhdan's Property (T), $B_1(G)$ is the Fourier--Stieltjes algebra $B(G)$, and $\tilde{A}_c(G)$ is the full C${}^*$-algebra $C^*(G)$. The equivalence in that case is given by \cite[Lemma 4.3]{HaKndL}.

Since the inclusions
\begin{align*}
B(G)\hookrightarrow B_c(G)\hookrightarrow M_0A(G)
\end{align*}
are weak*-weak*-continuous, we have the following implications:
\begin{align}\label{impl_T_c}
\text{Property (T}^*\text{)}\implies\text{Property (T)}_c\implies\text{Property (T)}.
\end{align}
Both these implications are strict; see Corollary \ref{Cor_str_imp}.

Let us mention that Property (T${}^*$) was generalised even further by de Laat and Zadeh \cite{dLaZad}. They defined Property (T${}_{X}^*$) in similar fashion, where $X(G)$ is a suitable subspace of $\operatorname{WAP}(G)$. In that language, Property (T)${}_c$ is exactly Property (T${}_X^*$) for $X=B_c$.

From the definition, we see that Property (T)${}_{c_2}$ implies Property (T)${}_{c_1}$ for $c_2\geq c_1 \geq 1$. This can also be deduced from the fact that inclusion $B_{c_1}(G)\hookrightarrow B_{c_2}(G)$ is weak*-weak*-continuous. This leads us to define, for every locally compact group $G$,
\begin{align*}
c_{\operatorname{ub}}(G)=\inf\{c\geq 1\ :\ G\text{ does not have Property (T)}_c\} \in [1,\infty],
\end{align*}
with the convention that $\inf\varnothing=\infty$. If $G$ satisfies Property (T${}^*$), then $c_{\operatorname{ub}}(G)=\infty$. The same holds if $G$ satisfies Lafforgue's Strong Property (T); see \cite{Laf} and the characterisation given by Corollary \ref{Cor_DN}. Our second result concerns the stability of Property (T)${}_c$ under various group constructions and relations.

\begin{thm}\label{Thm_c_ub}
The constant $c_{\operatorname{ub}}$ satisfies the following:
\begin{itemize}
\item[a)] For every $G$, $c_{\operatorname{ub}}(G)>1$ if and only if $G$ has Property (T).
\item[b)] For all $G_1,G_2$, $c_{\operatorname{ub}}(G_1\times G_2)=\min\{c_{\operatorname{ub}}(G_1),c_{\operatorname{ub}}(G_2)\}$.
\item[c)] If $H$ is a closed normal subgroup of $G$, then
\begin{align*}
\min\{c_{\operatorname{ub}}(G/H),c_{\operatorname{ub}}(H)\}\leq c_{\operatorname{ub}}(G)\leq c_{\operatorname{ub}}(G/H).
\end{align*}
\item[d)] If $\Gamma$ is a lattice in $G$, then $c_{\operatorname{ub}}(\Gamma)=c_{\operatorname{ub}}(G)$.
\item[e)] If $G$ is unitarisable, then $c_{\operatorname{ub}}(G)\in\{1,\infty\}$.
\end{itemize}
\end{thm}

The definitions of lattices and unitarisability are given in Sections \ref{Subsec_latt} and \ref{Sec_c_ub} respectively. 

Property (T) is known to be stable under measure equivalence; see \cite[Corollary 1.4]{Fur2}. This result was extended in \cite{IsPeRu} to von Neumann equivalence. Two countable groups $\Gamma$ and $\Lambda$ are \textit{von Neumann equivalent} ($\Gamma\sim_{\operatorname{vNE}}\Lambda$) if there is a von Neumann algebra $\mathcal{M}$ with a semi-finite, normal, faithful trace $\operatorname{Tr}$, and commuting, trace-preserving actions $\Gamma,\Lambda\curvearrowright(\mathcal{M},\operatorname{Tr})$ such that each action has a finite-trace fundamental domain; see Section \ref{Sec_vNE} for details. We prove the following.

\begin{thm}\label{Thm_vNE}
Let $\Gamma$ and $\Lambda$ be countable groups such that $\Gamma\sim_{\operatorname{vNE}}\Lambda$. Then $c_{\operatorname{ub}}(\Gamma)=c_{\operatorname{ub}}(\Lambda)$.
\end{thm}

As shown in \cite{IsPeRu}, von Neumann equivalence is implied by both measure equivalence and W${}^*$-equivalence. Two groups $\Gamma,\Lambda$ are  \textit{W${}^*$-equivalent} if their von Neumann algebras are isomorphic. A direct consequence of Theorem \ref{Thm_vNE} is the following.

\begin{cor}
Let $\Gamma$ and $\Lambda$ be countable groups such that their von Neumann algebras $L\Gamma, L\Lambda$ are isomorphic. Then $c_{\operatorname{ub}}(\Gamma)=c_{\operatorname{ub}}(\Lambda)$.
\end{cor}

Finally, we focus on examples. A natural question is whether there exist groups with $1 < c_{\operatorname{ub}}(G) < \infty$. In other words, are there Property (T) groups that do not satisfy Property (T)${}_c$ for $c$ big enough? We answer this question in the positive using Dooley's work on representations of rank 1 Lie groups \cite{Doo}. For $G=\operatorname{Sp}(n,1)$ ($n\geq 2$) or $G=F_{4,-20}$, let $\boldsymbol{\Lambda}(G)$ be the Cowling--Haagerup constant, i.e.
\begin{align}\label{CHctt}
\boldsymbol{\Lambda}(G)=\begin{cases}
2n-1, & \text{if }G=\operatorname{Sp}(n,1),\\
21, & \text{if }G=F_{4,-20}.
\end{cases}
\end{align}
We refer the reader to \cite[\S I.17]{Kna} and \cite{Val3} for the precise definitions of these groups, and to \cite{CowHaa} for the meaning of $\boldsymbol{\Lambda}(G)$. It turns out that this constant provides a bound for $c_{\operatorname{ub}}(G)$.

\begin{thm}\label{Thm_rank1}
For $G=\operatorname{Sp}(n,1)$ ($n\geq 2$) and $G=F_{4,-20}$, we have
\begin{align*}
1 < c_{\operatorname{ub}}(G) \leq \boldsymbol{\Lambda}(G).
\end{align*}
\end{thm}

The fact that these groups have Property (T) has been known for a long time; see \cite{Kos}. The new element in Theorem \ref{Thm_rank1} is the inequality $c_{\operatorname{ub}}(G) \leq \boldsymbol{\Lambda}(G)$.

A very important characterisation of Property (T) is given by the Delorme--Guichardet theorem, which says that, for $\sigma$-compact groups, Property (T) is equivalent to Property (FH): every isometric affine action on a Hilbert space has a fixed point; see \cite[Theorem 2.12.4]{BedlHVa}. Guichardet \cite{Gui} proved that Property (FH) implies Property (T) for $\sigma$-compact groups, and Delorme \cite{Del} proved the converse for every topological group. It is then natural to try to establish an analogous result for uniformly bounded representations. We say that $G$ has Property (FH)${}_c$ if every affine action whose linear part is in $\mathcal{R}_c(G)$ has a fixed point. As shown in \cite[\S 3.a]{BFGM}, Guichardet's argument holds in a far more general setting. This allows one to prove that Property (FH)${}_c$ implies Property (T)${}_c$ for all $c\geq 1$; see Section \ref{Sec_Guich} for details.

\begin{rmk}
We do not know if Property (T)${}_c$ implies Property (FH)${}_c$ for $c>1$.
\end{rmk}

We will say that a group $G$ has Property (PH)${}_c$ if it admits a (metrically) proper affine action on a Hilbert space with linear part in $\mathcal{R}_c(G)$. This is a strong negation of Property (FH)${}_c$. For $c=1$, this is nothing but the Haagerup property or a-(T)-menability; see \cite{CCJJV}. By a result of de Cornulier \cite{Cor}, for simple Lie groups, there is a dichotomy: they either satisfy Property (FH)${}_c$ or Property (PH)${}_c$. As a consequence of this fact, we obtain the following.

\begin{cor}\label{Cor_rank1}
Let $G$ be as in Theorem \ref{Thm_rank1}. Then $G$ satisfies Property (PH)${}_c$ for every $c>\boldsymbol{\Lambda}(G)$.
\end{cor}

Let us mention that Nishikawa \cite{Nis} proved that $\operatorname{Sp}(n,1)$ has Property (PH)${}_c$ for some $c>1$ using different methods.

In the final part of the article, we concentrate on giving lower bounds for $c_{\operatorname{ub}}$. More concretely, we use the work of Nowak \cite{Now} and de Laat--de la Salle \cite{dLadlS} to study the behaviour of $c_{\operatorname{ub}}(\Gamma)$, where $\Gamma$ is a random group in the Gromov density model and triangular model respectively. In the Gromov model $\mathcal{G}(n,l,d)$ with $d>\frac{1}{3}$, we have $c_{\operatorname{ub}}(\Gamma)\geq\sqrt{2}$ with overwhelming probability. In the triangular model $\mathcal{M}(n,d)$, $c_{\operatorname{ub}}(\Gamma)$ goes to infinity as $n\to\infty$, with a lower bound proportional to $n^{\frac{3d-1}{6}}$; see Section \ref{Sec_random} for details.

\subsection*{Organisation of the paper}
In Section \ref{Sec_Bc_Ac}, we focus on the spaces $B_c(G)$ and $\tilde{A}_c(G)$ and prove all the preliminary results related to representations in $\mathcal{R}_c(G)$. Section \ref{Sec_(T)_c} is devoted to Property (T)${}_c$ and its characterisation in terms of the existence of the Kazhdan projection. With all this material, we prove Theorem \ref{Thm_equiv} in Section \ref{Sec_proof_Thm}. Section \ref{Sec_Guich} is devoted to the generalisation of Guichardet's theorem to our context. In Section \ref{Sec_stab}, we prove several stability properties, leading to the proof of Theorem \ref{Thm_c_ub} in Section \ref{Sec_c_ub}. In section \ref{Sec_vNE}, we focus on von Neumann equivalence and prove Theorem \ref{Thm_vNE}. Theorem \ref{Thm_rank1} is proved in Section \ref{Sec_Lie}. Section \ref{Sec_random} is devoted to random groups and the lower bounds for $c_{\operatorname{ub}}$. Finally, in Section \ref{Sec_ques}, we pose a few questions related to Property (T)${}_c$ that we consider interesting.

\subsection*{Acknowledgements}
I am grateful to Anthony Dooley and Mikael de la Salle, whose comments and suggestions helped greatly improve the contents of this paper.

\section{{\bf The spaces $B_c(G)$ and $\tilde{A}_c(G)$}}\label{Sec_Bc_Ac}

The aim of this section is to gather all the definitions and material needed in the proof of Theorem \ref{Thm_equiv}. Unless otherwise specified, $G$ will always denote a locally compact group endowed with a left Haar measure.

\subsection{Weakly almost periodic functions}
We start by discussing invariant means and weakly almost periodic functions; for a much more detailed treatment, we refer the reader to \cite[\S 3]{HaKndL} and the references therein. Let $X$ be a subspace of $L^\infty(G)$ containing the constant functions and closed under complex conjugation. A \textit{mean} on $X$ is a positive linear functional $m:X\to\C$ such that $m(1)=1$. This implies that
\begin{align*}
|m(\varphi)|\leq\|\varphi\|_\infty,
\end{align*}
for all $\varphi\in X$; see \cite[Proposition 3.2]{Pie}. We say that $m$ is left-invariant if
\begin{align*}
m(s\cdot \varphi)=m(\varphi),\quad\forall s\in G,\ \forall\varphi\in X,
\end{align*}
where $s\cdot\varphi(t)=\varphi(s^{-1}t)$. A group $G$ is said to be \textit{amenable} if there is a left-invariant mean on $L^\infty(G)$. Although this property is not satisfied by every locally compact group, there is a very interesting subalgebra of $L^\infty(G)$ that always carries a unique left-invariant mean. Let $C_b(G)$ be the space of bounded, continuous functions on $G$, endowed with the uniform norm. We denote by $C_{00}(G)$ the subspace of compactly supported functions, and by $C_0(G)$ the closed subspace of those that vanish at infinity. A function $\varphi\in C_b(G)$ is said to be \textit{weakly almost periodic} if the set
\begin{align*}
\{s\cdot\varphi\ :\ s\in G\}
\end{align*}
is relatively compact in the weak topology of $C_b(G)$. We denote by $\operatorname{WAP}(G)$ the algebra of weakly almost periodic functions on $G$. Ryll-Nardzewski \cite{Ryl} showed that $\operatorname{WAP}(G)$ has a unique left-invariant mean for every $G$. Moreover, the following was proved in \cite[Theorem 3.6]{HaKndL}.

\begin{thm}[Haagerup--Knudby--de Laat]\label{Thm_LIM}
Let $G$ be a locally compact group. Let $X$ be a subspace of $\operatorname{WAP}(G)$ containing the constant functions, and which is closed under left translation and complex conjugation. Then $X$ has a unique left-invariant mean $m$. Moreover, if $X$ is also closed under right translations, then $m$ is right-invariant. If $X$ is closed under inversion, then $m$ is inversion invariant.
\end{thm}

By inversion we mean $\check{\varphi}(s)=\varphi(s^{-1})$. Observe that, if $G$ is not compact, $X\cap C_0(G)$ is always contained in the kernel of $m$.

\subsection{Spaces of coefficients of uniformly bounded representations}\label{Subs_B_c}
Let $G$ be a locally compact group, $\mathcal{H}$ a Hilbert space, and $\mathcal{B}(\mathcal{H})$ the algebra of bounded operators on $\mathcal{H}$. By a \textit{representation} of $G$ on $\mathcal{H}$, we mean a homomorphism $\pi:G\to\mathcal{B}(\mathcal{H})$ with values in the invertible elements of $\mathcal{B}(\mathcal{H})$, which is continuous for the strong operator topology (SOT), meaning that, for every $\xi\in\mathcal{H}$, the map
\begin{align*}
s\in G \mapsto \pi(s)\xi\in\mathcal{H}
\end{align*}
is continuous for the norm topology. We say that $\pi$ is \textit{uniformly bounded} if
\begin{align*}
|\pi|=\sup_{s\in G}\|\pi(s)\| < \infty.
\end{align*}
A \textit{coefficient} of a uniformly bounded representation $\pi:G\to\mathcal{B}(\mathcal{H})$ is a function $\varphi\in C_b(G)$ for which there exist $\xi,\eta\in\mathcal{H}$ such that
\begin{align}\label{dec_phi}
\varphi(s)=\langle\pi(s)\xi,\eta\rangle,\quad\forall s\in G.
\end{align}
Following the notation of \cite{Pis2}, for $c\geq 1$, we define $B_c(G)$ as the space of coefficients of representations $\pi:G\to\mathcal{B}(\mathcal{H})$ with $|\pi|\leq c$. We endow $B_c(G)$ with the norm
\begin{align}\label{norm_B_c}
\|\varphi\|_{B_c(G)}=\inf\|\xi\|\|\eta\|,
\end{align}
where the infimum is taken over over all the decompositions as in \eqref{dec_phi}. Observe that
\begin{align}\label{norm_inf_Bc}
\|\varphi\|_\infty\leq c\|\varphi\|_{B_c(G)},\quad\forall\varphi\in B_c(G).
\end{align}

\begin{rmk}
We warn the reader that the notation $B_c(G)$ is conflicting with other notations present in the literature, and can therefore lead to confusions. The space $B_c(G)$ is a generalisation of the Fourier--Stieltjes algebra $B(G)$, i.e. the space of coefficients of unitary representations of $G$. But this algebra can be generalised in several different ways. For $p\in(1,\infty)$, Herz \cite{Her} defines $B_p(G)$ as a subspace of the multipliers of the Fig\`a-Talamanca--Herz algebra $A_p(G)$, which was later identified by Daws \cite{Daw} as the space of $p$-completely bounded multipliers of $A_p(G)$. On the other hand, Cowling \cite{Cow} denotes by $B_p(G)$ the whole algebra of multipliers of $A_p(G)$. In addition, Runde \cite{Run} defines $B_p(G)$ as the space of coefficients of isometric representations on $\operatorname{QSL}_p$-spaces. As a result, $B_p(G)$ may stand for four different objects. We hope that such confusions will be avoided by using the letter $c$ instead of $p$.
\end{rmk}

For $c\geq 1$, let $\mathcal{R}_c(G)$ be the class of uniformly bounded representations $\pi:G\to\mathcal{B}(\mathcal{H}_\pi)$, where $\mathcal{H}_\pi$ is a Hilbert space and $|\pi|\leq c$. From now on, for $\pi\in\mathcal{R}_c(G)$, we will always denote by $\mathcal{H}_\pi$ the associated Hilbert space. 

Recall that the space $M(G)$ of complex Radon measures on $G$ is a Banach algebra for the convolution product, and $L^1(G)$ is a two-sided ideal in $M(G)$; see \cite[\S 2.5]{Fol} for details. Every $\pi\in\mathcal{R}_c(G)$ extends to a unital Banach algebra morphism $\pi:M(G)\to\mathcal{B}(\mathcal{H}_\pi)$ given by
\begin{align}\label{pi(mu)}
\pi(\mu)=\int_G\pi(s)\, d\mu(s),
\end{align}
where this integral is well defined in the strong operator topology, i.e.
\begin{align*}
\pi(\mu)\xi=\int_G\pi(s)\xi\, d\mu(s),
\end{align*}
for every $\xi\in\mathcal{H}_\pi$. Observe that
\begin{align*}
\|\pi(\mu)\xi\|\leq c\|\mu\|\|\xi\|,\quad\forall\mu\in M(G),\ \forall\xi\in\mathcal{H}_\pi.
\end{align*}
When restricted to $L^1(G)$, \eqref{pi(mu)} becomes
\begin{align}\label{pi(f)}
\pi(f)=\int_G f(s)\pi(s)\, ds,\quad\forall f\in L^1(G),
\end{align}
where we use $ds$ to denote the integration with respect to the Haar measure of $G$. In particular,
\begin{align*}
\|\pi(f)\|\leq c\|f\|_1,\quad\forall f\in L^1(G).
\end{align*}
We say that $\pi$ is cyclic if there exists $\xi\in\mathcal{H}_\pi$ such that $\pi(L^1(G))\xi$ is dense in $\mathcal{H}_\pi$.

\begin{lem}\label{Lem_cyc_rep}
Every $\varphi\in B_c(G)$ is a coefficient of a cyclic representation in $\mathcal{R}_c(G)$. Moreover, the norm $\|\varphi\|_{B_c(G)}$ can be computed as in \eqref{norm_B_c} using only cyclic representations.
\end{lem}
\begin{proof}
If $\varphi$ is given by $\langle\pi(\,\cdot\,)\xi,\eta\rangle$ with $\pi\in\mathcal{R}_c(G)$, then $\mathcal{K}=\overline{\pi(L^1(G))\xi}$ is a $\pi$-invariant subspace of $\mathcal{H}_\pi$ because
\begin{align*}
\pi(s)\pi(f)\xi=\pi(\delta_s\ast f)\xi,
\end{align*}
for every $s\in G$ and every $f\in L^1(G)$. Moreover, $\xi$ belongs to $\mathcal{K}$. Indeed, if $(f_i)$ is an approximate identity in $L^1(G)$, for every $\zeta\in\mathcal{H}_\pi$, we have
\begin{align*}
\langle\pi(f_i)\xi,\zeta\rangle = \int_G f_i(s)\langle\pi(s)\xi,\zeta\rangle\, ds \longrightarrow \langle\xi,\zeta\rangle;
\end{align*}
see \cite[Proposition 2.44]{Fol}. Hence
\begin{align*}
\varphi(t)=\langle\sigma(t)\xi,\eta'\rangle_{\mathcal{K}},\quad\forall t\in G,
\end{align*}
where $\sigma$ is the restriction of $\pi$ to $\mathcal{K}$ and $\eta'$ is the projection of $\eta$ on $\mathcal{K}$. By definition, $\sigma$ is cyclic and it belongs to $\mathcal{R}_c(G)$. Finally, the statement about the norm follows from the fact that $\|\eta'\|\leq\|\eta\|$.
\end{proof}

The following result says that the SOT-continuity in the definition of $B_c(G)$ is not such a restrictive hypothesis. We essentially follow the proof of \cite[Theorem 3.2]{Haa}. 

\begin{lem}\label{Lem_G_d}
Let $G$ be a locally compact group, and let $G_{\text{d}}$ be the same group endowed with the discrete topology. For all $c\geq 1$,
\begin{align*}
B_c(G)=B_c(G_{\text{d}})\cap C_b(G).
\end{align*}
Moreover, for every $\varphi\in B_c(G)$,
\begin{align*}
\|\varphi\|_{B_c(G)}=\|\varphi\|_{B_c(G_{\text{d}})}.
\end{align*}
\end{lem}
\begin{proof}
First observe that, by definition, $B_c(G)$ is included in $B_c(G_{\text{d}})\cap C_b(G)$, and
\begin{align*}
\|\varphi\|_{B_c(G_{\text{d}})}\leq\|\varphi\|_{B_c(G)},
\end{align*}
for every $\varphi\in B_c(G)$. Now assume that $\varphi$ belongs to $B_c(G_{\text{d}})\cap C_b(G)$ and let $\varepsilon>0$. By (the proof of) Lemma \ref{Lem_cyc_rep}, there is a representation $\pi\in\mathcal{R}_c(G_{\text{d}})$ and $\xi,\eta\in\mathcal{H}_\pi$ such that
\begin{align}\label{|xi||eta|leq}
\|\xi\|\|\eta\|\leq\|\varphi\|_{B_c(G_{\text{d}})}+\varepsilon,
\end{align}
\begin{align*}
\varphi(st)=\langle\pi(t)\xi,\pi(s)^*\eta\rangle,\quad\forall s,t\in G,
\end{align*}
and the linear span of $\{\pi(t)\xi\ :\ t\in G\}$ is dense in $\mathcal{H}_\pi$. Let now $\mathcal{H}_0$ be the closed linear span of $\{\pi(s)^*\eta\ :\ s\in G\}$, and let $P$ be the orthogonal projection onto $\mathcal{H}_0$. We can define a representation on $\mathcal{H}_0$ by
\begin{align*}
\tilde{\pi}(t)Px=P\pi(t)x,\quad\forall t\in G,\ \forall x\in\mathcal{H}_\pi.
\end{align*}
To see that this is well defined, observe that, if $Px=0$, then
\begin{align*}
\langle P\pi(t)x,\pi(s)^*\eta\rangle=\langle x,\pi(st)^*\eta\rangle=0,\quad\forall s\in G,
\end{align*}
which shows that $P\pi(t)x=0$. The fact that $\tilde{\pi}(st)=\tilde{\pi}(s)\tilde{\pi}(t)$ is straightforward. Moreover, for all $t\in G$, $x\in\mathcal{H}_\pi$, $y\in\mathcal{H}_0$,
\begin{align*}
|\langle\tilde{\pi}(t)Px,y\rangle| &= |\langle Px,\pi(t)^*y\rangle|\\
&\leq c\|Px\|\|y\|.
\end{align*}
This shows that $|\tilde{\pi}|\leq c$. Furthermore,
\begin{align}\label{phi=tilde}
\langle\tilde{\pi}(t)P\xi,\tilde{\pi}(s)^*\eta\rangle=\varphi(st),\quad\forall s,t\in G.
\end{align}
In particular, the map
\begin{align*}
t\in G\mapsto\langle\tilde{\pi}(t)P\xi,\tilde{\pi}(s)^*\eta\rangle
\end{align*}
is continuous for all $s\in G$. By linearity, density, and the uniform boundedness of $\tilde{\pi}$, we get that
\begin{align*}
t\in G\mapsto\langle\tilde{\pi}(t)P\xi,y\rangle,
\end{align*}
is continuous for every $y\in\mathcal{H}_0$. By the same arguments, we obtain the continuity of
\begin{align*}
t\in G\mapsto\langle\tilde{\pi}(t)x,y\rangle,
\end{align*}
for all $x,y\in\mathcal{H}_0$. This shows that $\tilde{\pi}$ is continuous for the weak operator topology, which is equivalent to the continuity for the strong operator topology; see \cite[Theorem 2.8]{dLeGli} or \cite[Lemma 2.4]{BFGM}. Therefore $\tilde{\pi}$ belongs to $\mathcal{R}_c(G)$ and $\varphi$ belongs to $B_c(G)$. By \eqref{phi=tilde} and \eqref{|xi||eta|leq},
\begin{align*}
\|\varphi\|_{B_c(G)}\leq \|P\xi\| \|\eta\| \leq \|\xi\| \|\eta\|\leq \|\varphi\|_{B_c(G_{\text{d}})}+\varepsilon.
\end{align*}
Since $\varepsilon$ was arbitrary, we conclude that
\begin{align*}
\|\varphi\|_{B_c(G)}=\|\varphi\|_{B_c(G_{\text{d}})}.
\end{align*}
\end{proof}

\begin{lem}\label{Lem_Bc_WAP}
For every $c\geq 1$, $B_c(G)$ is a translation invariant subspace of $\operatorname{WAP}(G)$ which is closed under complex conjugation and inversion.
\end{lem}
\begin{proof}
The fact that $B_c(G)$ is contained in $\operatorname{WAP}(G)$ follows from \cite[Proposition 3.3]{HaKndL} and the inclusion $B_c(G)\subset M_0A(G)$; see \cite[Theorem 2.2]{dCaHaa}. See also \cite[Example 9]{BaRoSa}. If $\varphi\in B_c(G)$ is given by $\varphi(t)=\langle\pi(t)\xi,\eta\rangle$, then the identities
\begin{align*}
\varphi(st)&=\langle\pi(t)\xi,\pi(s)^*\eta\rangle\\
\varphi(ts)&=\langle\pi(t)\pi(s)\xi,\eta\rangle
\end{align*}
show that $B_c(G)$ is both left and right translation invariant. Moreover,
\begin{align*}
\overline{\varphi(t)}=(\overline{\pi}(t)\overline{\xi},\overline{\eta}),
\end{align*}
where $\overline{\xi},\overline{\eta}$ are the corresponding elements of $\xi,\eta\in\mathcal{H}$ in the complex conjugate Hilbert space $\overline{\mathcal{H}}$, which is endowed with the scalar product
\begin{align*}
(\overline{\xi},\overline{\eta})=\langle\eta,\xi\rangle,
\end{align*}
and the representation
\begin{align*}
\overline{\pi}(t)\overline{\xi}=\overline{\pi(t)\xi}.
\end{align*}
This shows that $B_c(G)$ is closed under complex conjugation. Finally, observe that
\begin{align*}
\check{\varphi}(t)=\varphi(t^{-1})=(\overline{\sigma}(t)\overline{\eta},\overline{\xi}),
\end{align*}
where $\sigma(t)=\pi(t^{-1})^*$ is the adjoint representation, and $\overline{\sigma}$ is defined as above. This representation might not be SOT-continuous, but we know that $\check{\varphi}$ is continuous on $G$. By Lemma \ref{Lem_G_d}, $\check{\varphi}$ belongs to $B_c(G)$. We conclude that $B_c(G)$ is closed under inversion.
\end{proof}

\begin{cor}\label{Cor_inv_mean}
For every $c\geq 1$, $B_c(G)$ carries a unique invariant mean. This mean is also inversion invariant.
\end{cor}
\begin{proof}
This follows directly from Theorem \ref{Thm_LIM} and Lemma \ref{Lem_Bc_WAP}. Observe that $B_c(G)$ contains the constants because the trivial representation belongs to $\mathcal{R}_c(G)$.
\end{proof}

\begin{rmk}
The inequality \eqref{norm_inf_Bc} shows that the unique invariant mean $m\in B_c(G)^*$ satisfies
\begin{align*}
\|m\|_{B_c(G)^*}\leq c.
\end{align*}
\end{rmk}

\subsection{The predual $\tilde{A}_c(G)$ of $B_c(G)$}
We describe here a predual space of $B_c(G)$. This space was studied by Pisier in \cite{Pis}; see also \cite{Pis2}. Recall that every $\pi\in\mathcal{R}_c(G)$ extends to a Banach algebra morphism $\pi:L^1(G)\to\mathcal{B}(\mathcal{H}_\pi)$. We define $\tilde{A}_c(G)$ as the completion of $L^1(G)$ for the norm
\begin{align}\label{norm_A_c}
\|f\|_{\tilde{A}_c(G)}=\sup_{\pi\in\mathcal{R}_c(G)}\|\pi(f)\|_{\mathcal{B}(\mathcal{H}_\pi)}.
\end{align}
In other words,
\begin{align*}
\|f\|_{\tilde{A}_c(G)}=\sup\left\{\left|\int_G \varphi(s) f(s)\, ds\right|\ :\ \varphi\in B_c(G),\ \|\varphi\|_{B_c(G)}\leq 1\right\}.
\end{align*}
To see that these two quantities are the same, take $\varphi=\langle\pi(\,\cdot\,)\xi,\eta\rangle$ and observe that
\begin{align*}
\int_G \varphi(s)f(s)\, ds = \langle\pi(f)\xi,\eta\rangle.
\end{align*}
This norm satisfies
\begin{align*}
\|f\|_{\tilde{A}_c(G)}\leq c\|f\|_1,\quad\forall f\in L^1(G).
\end{align*}
Observe that the convolution in $L^1(G)$ extends to a product in $\tilde{A}_c(G)$ for which it becomes a Banach algebra.

\begin{rmk}
With the notation of \cite{DruNow} and \cite{dlS}, $\tilde{A}_c(G)$ corresponds to $\mathcal{C}_{\mathcal{F}}(G)$ for the family $\mathcal{F}=\mathcal{R}_c(G)$. We chose to use Pisier's notation because he introduced $\tilde{A}_c(G)$ as a predual space for $B_c(G)$, and this was one of our main motivations for defining and studying Property (T)${}_c$.
\end{rmk}

Before proving that $\tilde{A}_c(G)^*=B_c(G)$, we need the following lemma, which is a simple adaptation of \cite[Proposition 13.3.4]{Dix} to our setting. We include the proof for the sake of completeness.

\begin{lem}\label{Lem_pi_cont}
Let $\tilde{\pi}:L^1(G)\to\mathcal{B}(\mathcal{H})$ be a Banach algebra morphism such that $\tilde{\pi}(L^1(G))\mathcal{H}$ is dense in $\mathcal{H}$. Then there exists a (SOT-continuous) uniformly bounded representation $\pi:G\to\mathcal{B}(\mathcal{H})$ with
\begin{align*}
|\pi|=\|\tilde{\pi}\|,
\end{align*}
such that
\begin{align*}
\tilde{\pi}(f)=\int_G f(s)\pi(s)\, ds,
\end{align*}
for all $f\in L^1(G)$.
\end{lem}
\begin{proof}
For each $s\in G$, let us define a linear map $\pi(s)$ on $\mathcal{H}_0=\tilde{\pi}(L^1(G))\mathcal{H}$ by
\begin{align*}
\pi(s)\tilde{\pi}(f)\xi=\tilde{\pi}(\delta_s\ast f)\xi,\quad\forall f\in L^1(G),\ \forall\xi\in\mathcal{H},
\end{align*}
where $\delta_s$ is the Dirac measure on $s$. Now let $(g_i)$ be an approximation of the identity on $L^1(G)$, and observe that
\begin{align*}
\|\pi(s)\tilde{\pi}(f)\xi-\tilde{\pi}(g_i\ast\delta_s)\tilde{\pi}(f)\xi\| \leq \|\tilde{\pi}\| \|\delta_s\ast f-g_i\ast\delta_s\ast f\|_1 \|\xi\|,
\end{align*}
which shows that $\pi(s)$ is the limit of $\tilde{\pi}(g_i\ast\delta_s)$ in the strong operator topology of $\mathcal{H}_0$. Since $\mathcal{H}_0$ is dense in $\mathcal{H}$, $\pi(s)$ has a unique extension to $\mathcal{H}$, and
\begin{align*}
\|\pi(s)\|\leq \|\tilde{\pi}\|
\end{align*}
because $\|\tilde{\pi}(g_i\ast\delta_s)\|\leq \|\tilde{\pi}\| \|g_i\|_1$. Let $e$ be the identity element of $G$. The fact that $(g_i)$ is an approximation of the identity shows that $\pi(e)$ is the identity operator. Moreover,
\begin{align*}
\pi(st)\tilde{\pi}(f)\xi=\tilde{\pi}(\delta_s\ast\delta_t\ast f)\xi=\pi(s)\pi(t)\tilde{\pi}(f)\xi,\quad\forall f\in L^1(G),\ \forall\xi\in\mathcal{H},
\end{align*}
which shows that $\pi(st)=\pi(s)\pi(t)$ on $\mathcal{H}_0$, and a fortiori on $\mathcal{H}$. We claim that $\pi$ is SOT-continuous. Indeed, for all $f\in L^1(G)$ and $\xi\in\mathcal{H}$,
\begin{align*}
\|\pi(s)\tilde{\pi}(f)\xi-\tilde{\pi}(f)\xi\|\leq \|\tilde{\pi}\| \|\delta_s\ast f-f\|_1 \|\xi\|,
\end{align*}
which tends to $0$ as $s\to e$; see e.g. \cite[Proposition 2.42]{Fol}. As before, by density, we conclude that $\pi$ is continuous for the strong operator topology of $\mathcal{B}(\mathcal{H})$. Finally, extend $\pi$ to $L^1(G)$ as in \eqref{pi(f)}. For all $f,g\in L^1(G)$, we have
\begin{align*}
f\ast g=\int_G f(s)(\delta_s\ast g)\, ds,
\end{align*}
where the right hand side is an $L^1(G)$-valued integral. Thus
\begin{align*}
\tilde{\pi}(f)\tilde{\pi}(g) &= \tilde{\pi}(f\ast g)\\
&= \int_G f(s)\tilde{\pi}(\delta_s\ast g)\, ds\\
&= \int_G f(s)\pi(s)\tilde{\pi}(g)\, ds\\
&= \pi(f)\tilde{\pi}(g).
\end{align*}
This shows that $\pi(f)=\tilde{\pi}(f)$ on $\mathcal{H}_0$ for all $f\in L^1(G)$. By density, we have $\pi=\tilde{\pi}$. Finally,
\begin{align*}
\|\tilde{\pi}(f)\|=\|\pi(f)\|\leq |\pi| \| f\|_1,\quad\forall f\in L^1(G),
\end{align*}
which shows that $|\pi|=\|\tilde{\pi}\|$.
\end{proof}

The following result was proved in \cite[Proposition 3.1]{Pis} for discrete groups. We extend it here to locally compact groups. We refer the reader to \cite{Hei} for details on ultraproducts of Banach spaces and operators.

\begin{prop}\label{Prop_duality}
For all $c\geq 1$, we have $\tilde{A}_c(G)^*=B_c(G)$ for the duality
\begin{align*}
\langle\varphi,f\rangle_{B_c,\tilde{A}_c}=\int_G \varphi(s)f(s)\, ds,\quad\forall\varphi\in B_c(G),\ \forall f\in L^1(G).
\end{align*}
\end{prop}
\begin{proof}
Let $\alpha:B_c(G)\to\tilde{A}_c(G)^*$ be the injective contraction given by
\begin{align*}
\langle\alpha(\varphi),f\rangle_{\tilde{A}_c^*,\tilde{A}_c}=\int_G \varphi(s)f(s)\, ds,\quad\forall\varphi\in B_c(G),\ \forall f\in L^1(G).
\end{align*}
Let $B$ be the open unit ball of $B_c(G)$. By the Hahn--Banach theorem and the definition of the norm of $\tilde{A}_c(G)$, we see that $\alpha(B)$ is $\sigma(\tilde{A}_c(G)^*,\tilde{A}_c(G))$-dense in the unit ball of $\tilde{A}_c(G)^*$. Now let $\psi$ be an element of the unit ball of $\tilde{A}_c(G)^*$, and let $(\varphi_i)_{i\in I}$ be a net in $B$ such that $\alpha(\varphi_i)$ converges to $\psi$ in the weak* topology. Since $I$ is a directed set, there is a filter $\mathcal{F}$ on $I$ given by all the sets containing a subset of the form $\{\ i\in I\mid\ i\geq i_0\}$ for some $i_0\in I$. Let $\mathcal{U}$ be an ultrafilter containing $\mathcal{F}$. For each $i\in I$, there exist $\pi_i\in\mathcal{R}_c(G)$ and $\xi_i,\eta_i\in\mathcal{H}_{\pi_i}$ such that
\begin{align*}
\varphi_i(s)=\langle\pi_i(s)\xi_i,\eta_i\rangle,\quad\forall s\in G,
\end{align*}
and $\|\xi_i\|\|\eta_i\| < 1$. Let $\mathcal{H}=(\mathcal{H}_{\pi_i})_{\mathcal{U}}$ be the ultraproduct of the spaces $\mathcal{H}_{\pi_i}$ with respect to the ultrafilter $\mathcal{U}$. Let us define $\pi:L^1(G)\to\mathcal{B}(\mathcal{H})$ by
\begin{align*}
\pi(f)=(\pi_i(f))_{\mathcal{U}},\quad\forall f\in L^1(G).
\end{align*}
Since every $\pi_i:L^1(G)\to\mathcal{B}(\mathcal{H}_i)$ is a Banach algebra morphism of norm at most $c$, so is $\pi$. Now let $\mathcal{K}$ be the closure of $\pi(L^1(G))\mathcal{H}$ in $\mathcal{H}$. We can view $\pi$ as a morphism from $L^1(G)$ to $\mathcal{B}(\mathcal{K})$. Furthermore, $\pi(L^1(G))\mathcal{K}$ is dense in $\mathcal{K}$. Indeed, let $x\in\mathcal{K}$, $\varepsilon>0$, and choose $f\in L^1(G)$, $x'\in\mathcal{H}$ such that
\begin{align*}
\|x-\pi(f)x'\|<\varepsilon.
\end{align*}
Now let $(g_j)$ be an approximation of the identity in $L^1(G)$, and observe that $\pi(g_j)\pi(f)x'\in\pi(L^1(G))\mathcal{K}$. Moreover,
\begin{align*}
\|x-\pi(g_j)\pi(f)x'\|&\leq \|x-\pi(f)x'\| + \|\pi(f)x'-\pi(g_j)\pi(f)x'\|\\
&\leq \varepsilon + c\|f-g_j\ast f\|_1 \|x'\|,
\end{align*}
which shows that $x$ is in the closure of $\pi(L^1(G))\mathcal{K}$. By Lemma \ref{Lem_pi_cont}, $\pi$ is given by a continuous representation $\pi:G\to\mathcal{B}(\mathcal{K})$ with $|\pi|\leq c$. Finally, define $\xi=(\xi_i)_{\mathcal{U}}$ and $\eta=(\eta_i)_{\mathcal{U}}$, and observe that, for all $f\in L^1(G)$,
\begin{align*}
\langle\psi,f\rangle_{\tilde{A}_c^*,\tilde{A}_c} &= \lim_{\mathcal{U}} \int_G \varphi_i(s)f(s)\, ds\\
&= \lim_{\mathcal{U}} \langle\pi_i(f)\xi_i,\eta_i\rangle\\
&= \langle\pi(f)\xi,\eta\rangle.
\end{align*}
Defining $\varphi=\langle\pi(\,\cdot\,)\xi,\eta\rangle$, this implies that $\psi=\alpha(\varphi)$, which shows that $\alpha$ is surjective and isometric.
\end{proof}

\subsection{A universal representation}\label{Subsec_Urep}
Now we will construct a representation $\pi_U\in\mathcal{R}_c(G)$ such that the norm of $\tilde{A}_c(G)$ can be computed using only $\pi_U$, i.e.
\begin{align*}
\|f\|_{\tilde{A}_c(G)}=\|\pi_U(f)\|,\quad\forall f\in L^1(G).
\end{align*}
We say that two representations $\pi,\rho\in\mathcal{R}_c(G)$ are equivalent if there is an isometry $V:\mathcal{H}_\pi\to\mathcal{H}_\rho$ such that
\begin{align*}
\rho(s)=V\pi(s)V^{-1},\quad\forall s\in G.
\end{align*}
Observe that
\begin{align*}
\langle\pi(s)\xi,\eta\rangle_{\mathcal{H}_\pi}=\langle\rho(s)V\xi,V\eta\rangle_{\mathcal{H}_\rho},\quad\forall\xi,\eta\in\mathcal{H}_\pi,
\end{align*}
which shows that $\varphi:G\to\C$ is a coefficient of $\pi$ if and only if it is a coefficient of $\rho$. Furthermore, for all $f\in L^1(G)$,
\begin{align*}
\|\pi(f)\|_{\mathcal{H}_\pi}=\|\rho(f)\|_{\mathcal{H}_\rho}.
\end{align*}
Let us denote by $\mathfrak{R}_c(G)$ the set of (equivalence classes of) cyclic representations in $\mathcal{R}_c(G)$. This is indeed a set; see the discussion in \cite[\S 4]{Run}. By all the facts mentioned above, the infimum and supremum in \eqref{norm_B_c} and \eqref{norm_A_c} can be taken over $\mathfrak{R}_c(G)$. Now define
\begin{align*}
\mathcal{H}_U=\bigoplus_{\pi\in\mathfrak{R}_c(G)}\mathcal{H}_\pi
\end{align*}
and $\pi_U:G\to\mathcal{B}(\mathcal{H}_U)$ by
\begin{align}\label{univ_rep}
\pi_U=\bigoplus_{\pi\in\mathfrak{R}_c(G)}\pi.
\end{align}

\begin{lem}
The representation $\pi_U$ belongs to $\mathcal{R}_c(G)$.
\end{lem}
\begin{proof}
For all $\xi=(\xi_\pi)_{\pi\in\mathfrak{R}_c(G)}\in\mathcal{H}_U$ and all $s\in G$,
\begin{align*}
\|\pi_U(s)\xi\|_{\mathcal{H}_U}^2=\sum_{\pi\in\mathfrak{R}_c(G)}\|\pi(s)\xi_\pi\|_{\mathcal{H}_\pi}^2
\leq c^2\sum_{\pi\in\mathfrak{R}_c(G)}\|\xi_\pi\|_{\mathcal{H}_\pi}^2 =c^2\|\xi\|_{\mathcal{H}_U}^2,
\end{align*}
which shows that $|\pi_U|\leq c$. It only remains to check the continuity for the strong operator topology. Fix $\xi=(\xi_\pi)_{\pi\in\mathfrak{R}_c(G)}\in\mathcal{H}_U$ and $\varepsilon>0$. Since $\|\xi\|_{\mathcal{H}_U}<\infty$, there is $F\subset\mathfrak{R}_c(G)$ finite such that
\begin{align*}
\sum_{\pi\in\mathfrak{R}_c(G)\setminus F}\|\xi_\pi\|_{\mathcal{H}_\pi}^2 < \frac{\varepsilon^2}{2(c+1)^2}.
\end{align*}
Now, for each $\pi\in F$, choose a neighbourhood of the identity $W_\pi\subset G$ such that
\begin{align*}
\|\pi(s)\xi_\pi-\xi_\pi\|_{\mathcal{H}_\pi}^2<\frac{\varepsilon^2}{2 |F|}, \quad\forall s\in W_\pi.
\end{align*}
Defining $W=\bigcap_{\pi\in F}W_\pi$, for all $s\in W$ we get
\begin{align*}
\|\pi_U(s)\xi-\xi\|_{\mathcal{H}_U}^2 &= \sum_{\pi\in F}\|\pi(s)\xi_\pi-\xi_\pi\|_{\mathcal{H}_\pi}^2 
+ \sum_{\pi\in\mathfrak{R}_c(G)\setminus F}\|\pi(s)\xi_\pi-\xi_\pi\|_{\mathcal{H}_\pi}^2\\
&< \sum_{\pi\in F}\frac{\varepsilon^2}{2 |F|} + \sum_{\pi\in\mathfrak{R}_c(G)\setminus F}(c+1)^2\|\xi_\pi\|_{\mathcal{H}_\pi}^2\\
&\leq \varepsilon^2.
\end{align*}
Therefore $\pi_U$ belongs to $\mathcal{R}_c(G)$.
\end{proof}

All these facts allow us to see that, for all $f\in L^1(G)$,
\begin{align}\label{norm_pi_U}
\|\pi_U(f)\|_{\mathcal{H}_U}=\sup_{\pi\in\mathfrak{R}_c(G)}\|\pi(f)\|_{\mathcal{H}_\pi}
=\sup_{\pi\in\mathcal{R}_c(G)}\|\pi(f)\|_{\mathcal{H}_\pi}=\|f\|_{\tilde{A}_c(G)}.
\end{align}

\section{{\bf Property (T)${}_c$ and Kazhdan projections}}\label{Sec_(T)_c}

We will now give the formal definition of Property (T)${}_c$ and characterise it in terms of Kazhdan pairs. After that, we will focus on Kazhdan projections, as defined by Dru\c{t}u and Nowak \cite{DruNow}. In order to prove Theorem \ref{Thm_equiv}, we will first need to characterise Property (T)${}_c$ by the existence of a Kazhdan projection in $\tilde{A}_c(G)$. This projection is the invariant mean that we are looking for; see \cite[Lemma 4.3]{HaKndL} for a proof of this identification in the case $c=1$. Again, throughout this section, $G$ denotes a locally compact group endowed with a left Haar measure.

\subsection{Definition of Property (T)${}_c$}
Let $\pi:G\to\mathcal{B}(\mathcal{H}_\pi)$ be a representation. We say that $\xi\in\mathcal{H}_\pi$ is an invariant vector for $\pi$ if
\begin{align*}
\pi(s)\xi=\xi,\quad\forall s\in G.
\end{align*}
If $\pi$ is uniformly bounded, this implies that
\begin{align*}
\pi(f)\xi=\left(\int_G f\right)\xi,\quad\forall f\in L^1(G).
\end{align*}
We denote by $\mathcal{H}_\pi^{\text{inv}}$ the subspace of invariant vectors. We say that $\pi$ almost has invariant vectors if, for every compact subset $Q\subset G$ and every $\varepsilon>0$, there exists $\xi\in\mathcal{H}_\pi$ such that
\begin{align*}
\sup_{s\in Q}\|\pi(s)\xi-\xi\|<\varepsilon\|\xi\|.
\end{align*}
Now let $c\geq 1$. We say that $G$ has Property (T)${}_c$ if, for any $\pi\in\mathcal{R}_c(G)$, if $\pi$ almost has invariant vectors, then $\mathcal{H}_\pi^{\text{inv}}\neq\{0\}$. In other words, if $\mathcal{H}_\pi^{\text{inv}}=\{0\}$, then there exist a compact subset $Q\subset G$ and $\kappa>0$ such that
\begin{align*}
\sup_{s\in Q}\|\pi(s)\xi-\xi\|\geq\kappa\|\xi\|,\quad\forall\xi\in\mathcal{H}_\pi.
\end{align*}
Observe that every $\pi\in\mathcal{R}_c(G)$ gives rise to a representation $\tilde{\pi}\in\mathcal{R}_c(G)$ on the quotient space $\mathcal{H}_\pi/\mathcal{H}_\pi^{\text{inv}}$. More precisely,
\begin{align}\label{pi_tilde}
\tilde{\pi}(s)[\xi]=[\pi(s)\xi],\quad\forall s\in G,\ \forall\xi\in\mathcal{H}_\pi.
\end{align}
Here $[\xi]$ denotes the equivalence class of $\xi$ in $\mathcal{H}_\pi/\mathcal{H}_\pi^{\text{inv}}$.

\begin{lem}
Let $\pi\in\mathcal{R}_c(G)$ and let $\tilde{\pi}$ be as in \eqref{pi_tilde}. Then $\tilde{\pi}$ does not have non-trivial invariant vectors.
\end{lem}
\begin{proof}
Let $\xi\in\mathcal{H}_\pi$ such that $[\xi]$ is $\tilde{\pi}$-invariant. Then $\pi(s)\xi-\xi$ is $\pi$-invariant for all $s\in G$. Hence, for all $n\geq 1$,
\begin{align*}
\pi(s^n)\xi-\xi &= \sum_{k=1}^n \pi(s^k)\xi-\pi(s^{k-1})\xi\\
&= \sum_{k=1}^n \pi(s^{k-1})(\pi(s)\xi-\xi)\\
&= n(\pi(s)\xi-\xi),
\end{align*}
which shows that
\begin{align*}
\|\pi(s)\xi-\xi\| \leq \frac{1}{n}\|\pi(s^n)\xi-\xi\|\leq\frac{c+1}{n}\|\xi\|.
\end{align*}
Since this holds for all $s\in G$ and all $n\geq 1$, we conclude that $\xi$ is $\pi$-invariant. In other words, $[\xi]=0$.
\end{proof}

Let $Q$ be a compact subset of $G$ and let $\kappa>0$. We say that $(Q,\kappa)$ is a Kazhdan pair for $\pi$ if
\begin{align}\label{Kazh_pair}
\sup_{s\in Q}\|\tilde{\pi}(s)[\xi]-[\xi]\|\geq\kappa\|[\xi]\|,\quad\forall\xi\in\mathcal{H}_\pi.
\end{align}
More generally, we say that $(Q,\kappa)$ is a Kazhdan pair for a family of representations if it is a Kazhdan pair for every representation in the family. By the discussion above, Property (T)${}_c$ is equivalent to the fact that every $\pi\in\mathcal{R}_c(G)$ has a Kazhdan pair. This can be improved as follows.

\begin{lem}
Let $G$ be a locally compact group and $c\geq 1$. The following are equivalent:
\begin{itemize}\label{Lem_char_T_c}
\item[(i)] The group $G$ has Property (T)${}_c$.
\item[(ii)] The family $\mathcal{R}_c(G)$ has a Kazhdan pair.
\item[(iii)] The universal representation $\pi_U$ as defined in \eqref{univ_rep} has a Kazhdan pair.
\end{itemize}
\end{lem}
\begin{proof}
The implications (ii)$\implies$(i)$\implies$(iii) follow directly from the definitions and the previous discussion. We only need to show that (iii) implies (ii). Let $(Q,\kappa)$ be a Kazhdan pair for $\pi_U$, and let $\pi\in\mathcal{R}_c(G)$ such that $\mathcal{H}_\pi^{\text{inv}}=\{0\}$. Take $\xi\in\mathcal{H}_\pi$. By restricting $\pi$ to the closure of $\pi(L^1(G))\xi$ as in the proof of Lemma \ref{Lem_cyc_rep}, we may assume that $\pi$ is cyclic. Thus it is a subrepresentation of $\pi_U$, which implies that
\begin{align*}
\sup_{s\in Q}\|\pi(s)\xi-\xi\|\geq\kappa\|\xi\|,\quad\forall\xi\in\mathcal{H}_\pi.
\end{align*}
Since this holds for every $\pi\in\mathcal{R}_c(G)$ with $\mathcal{H}_\pi^{\text{inv}}=\{0\}$ and every $\xi\in\mathcal{H}_\pi$, we conclude that $(Q,\kappa)$ is a Kazhdan pair for the whole family $\mathcal{R}_c(G)$.
\end{proof}

\subsection{Kazhdan projections \`a la Dru\c{t}u--Nowak}
In order to prove Theorem \ref{Thm_equiv}, we will need to characterise Property (T)${}_c$ in terms of the existence of a Kazhdan projection in $\tilde{A}_c(G)$. This was done in \cite{DruNow} in a much more general setting. For the reader's convenience, we reproduce some of their ideas here, focusing on the particular case of representations in $\mathcal{R}_c(G)$.

Let $\sigma:G\to\mathcal{B}(E)$ be an isometric representation on a reflexive Banach space. Then $E$ decomposes as
\begin{align}\label{dec_E}
E=E^{\text{inv}}\oplus E^{\text{comp}},
\end{align}
where
\begin{align*}
E^{\text{inv}}=\{\xi\in E\ :\ \sigma(s)\xi=\xi,\ \forall s\in G\},
\end{align*}
and $E^{\text{comp}}$ is the annihilator of the subspace of invariant vectors of the adjoint representation $\sigma^*:G\to\mathcal{B}(E^*)$ defined by $\sigma^*(s)=\sigma(s^{-1})^*$. The advantage of this decomposition is that $E^{\text{comp}}$ is also a $\sigma$-invariant closed subspace of $E$; see \cite[\S 2.3]{DruNow} for details. In particular, the restriction of $\sigma$ to $E^{\text{comp}}$ is an isometric representation without non-trivial invariant vectors. In this context, we say that $(Q,\kappa)$ is a Kazhdan pair for $\sigma$ if
\begin{align*}
\sup_{s\in Q}\|\sigma(s)\xi-\xi\|_E\geq\kappa\|\xi\|_E,\quad\forall\xi\in E^{\text{comp}}.
\end{align*}
This definition coincides with the one in \eqref{Kazh_pair} because we can identify $E^{\text{comp}}$ with $E/E^{\text{inv}}$. More generally, we say that $(Q,\kappa)$ is a Kazhdan pair for a family of isometric representations if it is a Kazhdan pair for every representation in the family.

For a family of isometric representations $\mathcal{F}$, we define $\mathcal{C}_{\mathcal{F}}(G)$ as the completion of $L^1(G)$ for the norm
\begin{align}\label{norm_C_F}
\|f\|_{\mathcal{C}_{\mathcal{F}}(G)}=\sup_{\sigma\in\mathcal{F}}\|\sigma(f)\|.
\end{align}
Let $\sigma:G\to\mathcal{B}(E)$ be a representation in $\mathcal{F}$. Then $\sigma$ extends to an algebra morphism $\sigma:\mathcal{C}_{\mathcal{F}}(G)\to\mathcal{B}(E)$. A \textit{Kazhdan projection} in $\mathcal{C}_{\mathcal{F}}(G)$ is a central idempotent $p$ such that, for every $\sigma:G\to\mathcal{B}(E)$ in $\mathcal{F}$, $\sigma(p)$ is the projection $E\to E^{\text{inv}}$ along $E^{\text{comp}}$.

Recall that the modulus of convexity of a Banach space $E$ is the function $\delta_E:[0,2]\to[0,1]$ defined by
\begin{align*}
\delta_E(t)=\inf\left\{1-\left\|\frac{\xi+\eta}{2}\right\|\ \mid\ \|\xi\|=\|\eta\|=1,\ \|\xi-\eta\|\geq t\right\}.
\end{align*}
We say that $E$ is uniformly convex if $\delta_E(t)>0$ for all $t>0$. More generally, we say that a family of Banach spaces $\mathcal{E}$ is uniformly convex if
\begin{align*}
\delta_{\mathcal{E}}(t)=\inf_{E\in\mathcal{E}}\delta_E(t) > 0
\end{align*}
for all $t>0$. The following was proved in \cite[Theorem 4.6]{DruNow}.

\begin{thm}[Dru\c{t}u--Nowak]\label{Thm_DN}
Let $G$ be a locally compact group and let $\mathcal{F}$ be a family of isometric representations of $G$ on a uniformly convex family of Banach spaces $\mathcal{E}$. Then $\mathcal{C}_{\mathcal{F}}(G)$ has a Kazhdan projection if and only if $\mathcal{F}$ has a Kazhdan pair. Moreover, in that case, the Kazhdan projection is the limit of a sequence $(g_n)$ in $C_{00}(G)$ such that $g_n\geq 0$ and $\|g_n\|_1=1$.
\end{thm}

The following is an application of Theorem \ref{Thm_DN} to the context of uniformly bounded representations on Hilbert spaces. For every $\pi\in\mathcal{R}_c(G)$, we will denote by $\pi^*$ the adjoint representation $\pi^*(s)=\pi(s^{-1})^*$.

\begin{cor}\label{Cor_DN}
Let $G$ be a locally compact group and $c\geq 1$. Then $G$ has Property (T)${}_c$ if and only if there exists a central idempotent $p\in\tilde{A}_c(G)$ such that, for every $\pi\in\mathcal{R}_c(G)$, $\pi(p)$ is the projection onto $\mathcal{H}_\pi^{\text{inv}}$ along $\left(\mathcal{H}_{\pi^*}^{\text{inv}}\right)^\perp$. Moreover $p$ is the limit of a sequence $(g_n)$ in $C_{00}(G)$ such that $g_n\geq 0$ and $\|g_n\|_1=1$.
\end{cor}
\begin{proof}
By Lemma \ref{Lem_char_T_c}, $G$ has Property (T)${}_c$ if and only if $\mathcal{R}_c(G)$ has a Kazhdan pair. Hence, by Theorem \ref{Thm_DN}, we only need to identify $\mathcal{R}_c(G)$ with a family of isometric representations on a uniformly convex family of Banach spaces. For $\pi\in\mathcal{R}_c(G)$, we let $E_\pi$ be the space $\mathcal{H}_\pi$ endowed with the equivalent norm
\begin{align*}
\|\xi\|_{E_\pi}=\sup_{s\in G}\|\pi(s)\xi\|_{\mathcal{H}_\pi}.
\end{align*}
Then $\pi$ is an isometric representation on $E_\pi$. Moreover, as shown in \cite[Proposition 2.3]{BFGM}, the modulus of convexity of $E_\pi$ satisfies
\begin{align*}
\delta_{E_\pi}(t)\geq \delta_{\mathcal{H}_\pi}(t/c),\quad\forall t\in[0,2].
\end{align*}
On the other hand, by the parallelogram identity,
\begin{align*}
\delta_{\mathcal{H}_\pi}(t)=1-\left(1-\frac{t^2}{4}\right)^{\frac{1}{2}},\quad\forall t\in[0,2],
\end{align*}
which shows that the family $\{E_\pi\}_{\pi\in\mathcal{R}_c(G)}$ is uniformly convex. Observe that, if $(Q,\kappa)$ is a Kazhdan pair for $\mathcal{R}_c(G)$, then $(Q,\kappa/c)$ is a Kazhdan pair for $\mathcal{F}=\{(E_\pi,\pi)\}_{\pi\in\mathcal{R}_c(G)}$, and vice versa. Finally, from the definition of the norm in $E_\pi$, we see that
\begin{align*}
\|f\|_{\tilde{A}_c(G)}\leq \|f\|_{\mathcal{C}_{\mathcal{F}}(G)} \leq c\|f\|_{\tilde{A}_c(G)},\quad\forall f\in L^1(G),
\end{align*}
which shows that the algebras $\tilde{A}_c(G)$ and $\mathcal{C}_{\mathcal{F}}(G)$ are the same. Observing that the decomposition $E_{\pi}=E_{\pi}^{\text{inv}}\oplus E_{\pi}^{\text{comp}}$ can be written in this context as
\begin{align}\label{dec_H_pi}
\mathcal{H}_\pi=\mathcal{H}_\pi^{\text{inv}}\oplus\left(\mathcal{H}_{\pi^*}^{\text{inv}}\right)^\perp,
\end{align}
the result follows directly from Theorem \ref{Thm_DN}.
\end{proof}

Henceforth, we will refer to the idempotent in Corollary \ref{Cor_DN} as the Kazhdan projection of $\tilde{A}_c(G)$.

\section{{\bf Equivalence between Kazhdan projections and invariant means}}\label{Sec_proof_Thm}

Now we are ready to prove Theorem \ref{Thm_equiv}. We begin with the implication (i)$\implies$(ii).

\begin{lem}\label{Lem_Kazh->m}
Let $G$ be a locally compact group and $c\geq 1$. If $G$ has Property (T)${}_c$, then the Kazhdan projection of $\tilde{A}_c(G)$ and the unique invariant mean on $B_c(G)$ coincide. In particular, this mean is weak*-continuous.
\end{lem}
\begin{proof}
By Corollary \ref{Cor_DN}, $\tilde{A}_c(G)$ has a Kazhdan projection $p$. Let $m$ the unique invariant mean on $B_c(G)$. By the uniqueness of $m$, we only need to show that the (weak*-continuous) map
\begin{align*}
\varphi\in B_c(G) \mapsto \langle\varphi,p\rangle_{B_c,\tilde{A}_c}
\end{align*}
is a left-invariant mean. Since $p$ is the limit of a sequence of non-negative functions $(g_n)$ with $\|g_n\|_1=1$, we have
\begin{align*}
\langle 1,p\rangle_{B_c,\tilde{A}_c} = \lim_n\int_G g_n(s)\, ds = 1.
\end{align*}
Similarly, we have $\langle\varphi,p\rangle\geq 0$ for every $\varphi\geq 0$. Finally, let $\varphi\in B_c(G)$ be given by $\varphi=\langle\pi(\cdot)\xi,\eta\rangle$. For all $s,t\in G$, we have
\begin{align*}
s\cdot\varphi(t)=\langle\pi(t)\xi,\pi(s^{-1})^*\eta\rangle.
\end{align*}
Therefore
\begin{align*}
\langle s\cdot\varphi,p\rangle &= \langle\pi(p)\xi,\pi(s^{-1})^*\eta\rangle\\
&= \langle\pi(s^{-1})\pi(p)\xi,\eta\rangle\\
&= \langle\pi(p)\xi,\eta\rangle\\
&= \langle \varphi,p\rangle.
\end{align*}
Here we used the fact that $\pi(p)\xi$ is an invariant vector.
\end{proof}

Now we deal with the implication (ii)$\implies$(i) in Theorem \ref{Thm_equiv}. The following lemma was proved in \cite[Lemma 5.11]{HaKndL} for the algebra of completely bounded Fourier multipliers $M_0A(G)$, but it holds in a much more general setting; see \cite[Lemma 5.5]{dLaZad}. We repeat the argument here for the space $B_c(G)$ for the sake of completeness.

\begin{lem}\label{Lem_app_m}
Let $G$ be a locally compact group and $c\geq 1$. If the unique invariant mean $m\in B_c(G)^*$ is weak*-continuous, then there exists a sequence $(g_n)$ of non-negative, continuous, compactly supported functions on $G$ with $\|g_n\|_1=1$ such that
\begin{align*}
\|g_n-m\|_{\tilde{A}_c(G)}\to 0.
\end{align*}
\end{lem}
\begin{proof}
We may extend $m$ to a state on $L^\infty(G)$. Since the normal states are dense in the state space of $L^\infty(G)$, there is a net $(f_\alpha)$ in $L^1(G)$ such that $f_\alpha\geq 0$, $\|f_\alpha\|_1=1$, and $f_\alpha\to m$ in $\sigma(L^\infty(G)^*,L^\infty(G))$. This implies that $f_\alpha\to m$ in $\sigma(\tilde{A}_c(G),B_c(G))$. Taking convex combinations, we find a sequence $(f_n)$ with $f_n\geq 0$ and $\|f_n\|_1=1$ such that $\|f_n-m\|_{\tilde{A}_c(G)}\to 0$. Finally, we choose $g_n\in C_{00}(G)$ such that $g_n\geq 0$ and $\|g_n\|_1=1$ and $\|g_n-f_n\|_1<\frac{1}{n}$. Since the inclusion $L^1(G)\hookrightarrow\tilde{A}_c(G)$ is continuous of norm at most $c$, we get
\begin{align*}
\|g_n-m\|_{\tilde{A}_c(G)} &\leq \|g_n-f_n\|_{\tilde{A}_c(G)} + \|f_n-m\|_{\tilde{A}_c(G)}\\
&\leq \frac{c}{n} + \|f_n-m\|_{\tilde{A}_c(G)} \xrightarrow[n\to\infty]{} 0.
\end{align*}
\end{proof}

\begin{lem}\label{Lem_m->Kazh}
Let $G$ be a locally compact group and $c\geq 1$. If the unique invariant mean on $B_c(G)$ is weak*-continuous, then this mean is the Kazhdan projection of $\tilde{A}_c(G)$. In particular, $G$ has Property (T)${}_c$.
\end{lem}
\begin{proof}
Let $m$ be the invariant mean on $B_c(G)$. Then, by hypothesis, $m$ belongs to $\tilde{A}_c(G)$. Moreover, by Lemma \ref{Lem_app_m}, $m$ is the limit of a sequence of compactly supported, continuous probability measures $(g_n)_{n\in\N}$. Now fix $\pi\in\mathcal{R}_c(G)$, and recall that we can decompose $\mathcal{H}_\pi$ as in \eqref{dec_H_pi}:
\begin{align*}
\mathcal{H}_\pi=\mathcal{H}_\pi^{\text{inv}}\oplus\left(\mathcal{H}_{\pi^*}^{\text{inv}}\right)^\perp.
\end{align*}
We need to show that $\pi(m)$ is the projection onto $\mathcal{H}_\pi^{\text{inv}}$ such that $\operatorname{Ker}(\pi(m))=\left(\mathcal{H}_{\pi^*}^{\text{inv}}\right)^\perp$. First observe that, for every $\xi\in\mathcal{H}_\pi^{\text{inv}}$,
\begin{align*}
\pi(m)\xi&=\lim_n\int_G g_n(s)\pi(s)\xi\, ds\\
&=\lim_n\left(\int_G g_n(s)\, ds\right)\xi\\
&=\xi.
\end{align*}
On the other hand, for $\xi,\eta\in\mathcal{H}_\pi$, we can define $\varphi=\langle\pi(\cdot)\xi,\eta\rangle$. Then, for all $s\in G$,
\begin{align*}
\langle\pi(s)\pi(m)\xi,\eta\rangle &= \langle s^{-1}\cdot\varphi,m\rangle_{B_c,\tilde{A}_c}\\
&= m(s^{-1}\cdot\varphi)\\
&= m(\varphi)\\
&= \langle\pi(m)\xi,\eta\rangle.
\end{align*}
Since this holds for every $\xi,\eta\in\mathcal{H}_\pi$, we conclude that
\begin{align}\label{pi(s)pi(p)}
\pi(s)\pi(m)=\pi(m),\quad\forall s\in G.
\end{align}
This shows that $\pi(m)$ is a projection onto $\mathcal{H}_\pi^{\text{inv}}$. The same argument shows that $\pi^*(m)$ is a projection onto $\mathcal{H}_{\pi^*}^{\text{inv}}$. Recall now that $m$ is inversion invariant. Hence we also have $\check{g}_n\to m$, where
\begin{align*}
\check{g}_n(s)=g_n(s^{-1})\Delta(s^{-1}),\quad\forall s\in G,
\end{align*}
and $\Delta$ is the modular function; see \cite[\S 2.4]{Fol} for details. Thus, for every $\eta\in\mathcal{H}_{\pi}$,
\begin{align*}
\pi^*(m)\eta&=\lim_n\int_G \check{g}_n(s)\pi^*(s)\eta\, ds\\
&=\lim_n\int_G g_n(s^{-1})\Delta(s^{-1})\pi(s^{-1})^*\eta\, ds\\
&=\lim_n\int_G g_n(s)\pi(s)^*\eta\, ds\\
&=\pi(m)^*\eta.
\end{align*}
If we now take $\xi\in\left(\mathcal{H}_{\pi^*}^{\text{inv}}\right)^\perp$, then
\begin{align*}
\langle\pi(m)\xi,\eta\rangle &= \langle\xi,\pi(m)^*\eta\rangle\\
&=\langle\xi,\pi^*(m)\eta\rangle\\
&=0.
\end{align*}
Therefore $\pi(m)$ is the projection onto $\mathcal{H}_\pi^{\text{inv}}$ along $\left(\mathcal{H}_{\pi^*}^{\text{inv}}\right)^\perp$. By corollary  \ref{Cor_DN}, we conclude that $G$ has Property (T)${}_c$.
\end{proof}

\section{{\bf A generalisation of Guichardet's theorem}}\label{Sec_Guich}

In this section, we show that Property (FH)${}_c$ implies Property (T)${}_c$. As discussed in the introduction, this was essentially done in \cite{BFGM} for representations on general Banach spaces. Nevertheless, we believe it pertinent to clearly state how their results translate into our new terminology. Again, we will use a renorming trick in order to deal only with isometric representations. For $c\geq 1$, we will denote by $\mathcal{E}_c$ the class of Banach spaces $E$ for which there is a Hilbert space $\mathcal{H}$ and a Banach space isomorphism $V:E\to\mathcal{H}$ such that
\begin{align}\label{normV}
\|V\|\|V^{-1}\|\leq c.
\end{align}
Observe that all the spaces in this class are reflexive. The following lemma gives a connection between representations in $\mathcal{R}_c(G)$ and isometric representations on spaces in the class $\mathcal{E}_c$. Notice that we already used this identification implicitly in the proof of Corollary \ref{Cor_DN}. Its proof is elementary, but it is worth stating it precisely.

\begin{lem}\label{Lem_corr_rep}
Let $G$ be a locally compact group and let $c\geq 1$.
\begin{itemize}
\item[a)] Let $\pi\in\mathcal{R}_c(G)$. Then the space $E=\mathcal{H}_\pi$, endowed with the norm
\begin{align*}
\|\xi\|_E=\sup_{s\in G}\|\pi(s)\xi\|_{\mathcal{H}_\pi},
\end{align*}
belongs to $\mathcal{E}_c$. Moreover $\pi$ defines an isometric representation on $E$.
\item[b)] Let $\rho:G\to\mathcal{B}(E)$ be an isometric representation on $E\in\mathcal{E}_c$, and let $V:E\to\mathcal{H}$ be as in \eqref{normV}. Then the representation $\pi=V\rho(\,\cdot\,)V^{-1}$ belongs to $\mathcal{R}_c(G)$.
\end{itemize}
\end{lem}
\begin{proof}
For (a), we define $V$ as the identity. Then $\|V\|\leq 1$, $\|V^{-1}\|\leq c$, and $\pi$ is clearly isometric on $E$. For (b), simply observe that $\|\pi(s)\|\leq\|V\|\|V^{-1}\|$. The continuity is not an issue, since in both cases the spaces are isomorphic.
\end{proof}

Now we recall the definitions of Property $(T_E)$ and Property $(F_E)$ from \cite{BFGM}. Let $\rho:G\to\mathcal{B}(E)$  be an isometric representation of a locally compact group $G$ on a Banach space $E$, and let $E^{\text{inv}}$ denote the subspace of $\rho$-invariant vectors. Then there is an isometric representation $\tilde{\rho}:G\to\mathcal{B}(E/E^{\text{inv}})$ given by
\begin{align*}
\tilde{\rho}(s)[\xi]=[\rho(s)\xi],\quad\forall s\in G,\ \forall\xi\in E.
\end{align*}
For a fixed Banach space $E$, we say that $G$ has Property $(T_E)$ if, for every isometric representation $\rho$ on $E$, the representation $\tilde{\rho}:G\to\mathcal{B}(E/E^{\text{inv}})$ does not almost have invariant vectors. This property may also be expressed in terms of Kazhdan pairs as in \eqref{Kazh_pair}.

Let now $G\curvearrowright^\sigma E$ be an affine action. This means that
\begin{align*}
\sigma(s)x=\pi(s)x+b(s),\quad\forall s\in G,\ \forall x\in E,
\end{align*}
where $\pi:G\to\mathcal{B}(E)$ is a representation and $b:G\to E$ satisfies the cocycle identity:
\begin{align*}
b(st)=\pi(s)b(t)+b(s),\quad\forall s,t\in G.
\end{align*}
We will only be interested in continuous actions, in the sense that
\begin{align*}
s\mapsto\sigma(s)x
\end{align*}
is continuous for all $x\in E$. Observe that this is equivalent to $\pi$ and $b$ being continuous for the strong operator topology and norm topology respectively. We say that $G$ has Property $(F_E)$ if every isometric affine action $G\curvearrowright E$ has a fixed point.

We say that $G$ has Property (FH)${}_c$ if every affine action of $G$ on a Hilbert space, whose linear part belongs to $\mathcal{R}_c(G)$, has a fixed point. For this kind of actions, this is equivalent to having bounded orbits; see \cite[Lemma 2.14]{BFGM}.

\begin{prop}\label{Prop_Guich_c}
Let $G$ be a locally compact, $\sigma$-compact group, and let $c\geq 1$. If $G$ has Property (FH)${}_c$, then it has Property (T)${}_c$.
\end{prop}
\begin{proof}
Let $E\in\mathcal{E}_c$ and let $\sigma$ be an isometric affine action of $G$ on $E$ given by
\begin{align*}
\sigma(s)x=\rho(s)x+b(s),\quad\forall s\in G,\ \forall x\in E.
\end{align*}
By Lemma \ref{Lem_corr_rep}, there is $V:E\to\mathcal{H}$ such that the action
\begin{align*}
\tilde{\sigma}(s)\xi=V\rho(s)V^{-1}\xi+Vb(s),\quad\forall s\in G,\ \forall \xi\in\mathcal{H},
\end{align*}
has linear part in $\mathcal{R}_c(G)$. Since $G$ has Property (FH)${}_c$, this action has a fixed point, which implies that $\sigma$ has a fixed point. In other words, $G$ has Property $(F_E)$ for every $E\in\mathcal{E}_c$. By \cite[Theorem 1.3(1)]{BFGM}, this implies that $G$ has Property $(T_E)$ for every $E\in\mathcal{E}_c$. Again, by Lemma \ref{Lem_corr_rep}, every $\pi\in\mathcal{R}_c(G)$ has a Kazhdan pair; see the proof of Corollary \ref{Cor_DN}. Therefore $G$ has Property (T)${}_c$.
\end{proof}

\section{{\bf Stability}}\label{Sec_stab}

In this section, we concentrate on the stability properties (b), (c) and (d) in Theorem \ref{Thm_c_ub}. Thanks to Theorem \ref{Thm_equiv} and Corollary \ref{Cor_DN}, we can take any of the three equivalent conditions as our definition of Property (T)${}_c$. And that is exactly what we will do, depending on which one is more convenient for each proof.

\subsection{Quotients}
Let us fix a locally compact group $G$, a closed normal subgroup $H$ of $G$, and $c\geq 1$. Let $q:G\to G/H$ denote the quotient map. 

\begin{lem}\label{Lem_pioq}
For every $\pi\in\mathcal{R}_c(G/H)$, the representation $\pi\circ q$ belongs to $\mathcal{R}_c(G)$.
\end{lem}
\begin{proof}
The fact that $q$ is a continuous homomorphism shows that $\pi\circ q$ is indeed a continuous representation. Furthermore, by definition, $|\pi\circ q|=|\pi|$, which shows that $\pi\circ q$ belongs to $\mathcal{R}_c(G)$.
\end{proof}

In this situation, we are dealing with three locally compact groups: $G$, $H$ and $G/H$. Fixing the Haar measures on two of them, we can normalise the Haar measure on the third so that
\begin{align*}
\int_{G/H}\int_H f(sh)\, dh\, d\dot{s} = \int_G f(s)\, ds,
\end{align*}
for all $f\in C_{00}(G)$; see \cite[\S 3.3]{ReiSte}. Here $\dot{s}=q(s)$. This implies that the map $T:L^1(G)\to L^1(G/H)$ given by
\begin{align}\label{T:L1}
Tf(q(s))=\int_H f(sh)\, dh,\quad\forall f\in L^1(G),
\end{align}
is a contraction.

\begin{lem}\label{Lem_T:A}
Let $G$ and $H$ be as above. The map \eqref{T:L1} extends to a contraction $T:\tilde{A}_c(G)\to\tilde{A}_c(G/H)$.
\end{lem}
\begin{proof}
Let $\pi\in\mathcal{R}_c(G/H)$. For every $f\in L^1(G)$,
\begin{align*}
\pi(Tf) &= \int_{G/H}\left(\int_H f(sh)\, dh\right)\pi(\dot{s})\, d\dot{s}\\
&=\int_G f(s)\pi(q(s))\, ds\\
&=\tilde{\pi}(f),
\end{align*}
where $\tilde{\pi}=\pi\circ q$. This shows that $\|\pi(Tf)\|=\|\tilde{\pi}(f)\|$. By Lemma \ref{Lem_pioq}, we conclude that
\begin{align*}
\|Tf\|_{\tilde{A}_c(G/H)}\leq\|f\|_{\tilde{A}_c(G)}.
\end{align*}
\end{proof}

\begin{cor}\label{Cor_G->G/H}
Let $G$ and $H$ be as above. If $G$ has Property (T)${}_c$, then so does $G/H$.
\end{cor}
\begin{proof}
Let $p\in\tilde{A}_c(G)$ be the unique invariant mean on $B_c(G)$, and let $T$ be the map given by Lemma \ref{Lem_T:A}. We claim that $Tp$ is the invariant mean on $B_c(G/H)$. Indeed, a simple calculation shows that the dual map $T^*:B_c(G/H)\to B_c(G)$ is given by
\begin{align*}
T^*\varphi=\varphi\circ q,\quad\forall\varphi\in B_c(G/H).
\end{align*}
In particular,
\begin{align*}
\langle\varphi,Tp\rangle=\langle\varphi\circ q,p\rangle,\quad\forall\varphi\in B_c(G/H).
\end{align*}
With this formula, the positivity and left-invariance of $Tp$ follow from those of $p$. Finally,
\begin{align*}
\langle 1,Tp\rangle=\langle 1,p\rangle=1.
\end{align*}
Hence $Tp$ is the unique invariant mean on $B_c(G/H)$.
\end{proof}

\subsection{Extensions}
Again, we fix a locally compact group $G$, a closed normal subgroup $H$ of $G$, and $c\geq 1$. For convenience, let us denote by $\nu$ the Haar measure of $H$. For every $s\in G$ and every Borel subset $A\subset H$, we define
\begin{align*}
\nu_s(A)=\nu(s^{-1}As).
\end{align*}
Since $H$ is normal, for all $t\in H$,
\begin{align*}
\nu_s(tA)=\nu(s^{-1}tAs)=\nu(s^{-1}tss^{-1}As)=\nu(s^{-1}As)=\nu_s(A).
\end{align*}
This implies that $\nu_s$ is also a Haar measure. Hence there is a constant $\delta(s)>0$ such that $\nu_s=\delta(s)\nu$.

\begin{lem}\label{Lem_tilde(g_n)}
Let $G$ and $H$ be as above. Let $m\in B_c(H)^*$ be the unique invariant mean on $B_c(H)$, and let $(g_n)$ be a sequence in $L^1(H)$ converging to $m$. For every $s\in G$, the sequence $(g_n^{(s)})$ given by
\begin{align*}
g_n^{(s)}(t)=\delta(s)^{-1}g_n(sts^{-1}),\quad\forall t\in H,
\end{align*}
converges to $m$ in $B_c(H)^*$.
\end{lem}
\begin{proof}
Let $m_s\in B_c(H)^*$ be given by $m_s(\varphi)=m(\varphi_s)$, where
\begin{align*}
\varphi_s(t)=\varphi(s^{-1}ts),\quad t\in H.
\end{align*}
Observe that, if $\varphi$ is a coefficient of a representation $\pi$, then $\varphi_s$ is a coefficient of $\pi(s^{-1}\,\cdot\, s)$, which shows that $m_s$ is well defined. Then, for all $\varphi\in B_c(H)$,
\begin{align*}
\int_H\varphi(t)g_n^{(s)}(t)\, d\nu(t) &= \int_H\varphi_s(t)g_n(t)\, d\nu(t)\\
&\xrightarrow[n\to\infty]{} m_s(\varphi).
\end{align*}
Moreover, since $(g_n)$ is a Cauchy sequence in $B_c(H)^*$, and $\langle g_n^{(s)},\varphi\rangle=\langle g_n,\varphi_s\rangle$, $(g_n^{(s)})$ is a Cauchy sequence too. Therefore $g_n^{(s)}\to m_s$ in $B_c(H)^*$. By uniqueness, we only need to show that $m_s$ is an invariant mean on $B_c(H)$. It is a mean because $m$ is a mean. In order to show that it is invariant, observe that $(t\cdot\varphi)_s=(sts^{-1})\cdot\varphi_s$ for all $t\in H$. Thus
\begin{align*}
m_s(t\cdot\varphi)=m((t\cdot\varphi)_s)=m((sts^{-1})\cdot\varphi_s)=m(\varphi_s)=m_s(\varphi).
\end{align*}
We conclude that $m_s=m$.
\end{proof}

For the following lemma, we adapt the arguments of \cite[Lemma 1.7.5]{BedlHVa} to our setting.

\begin{lem}\label{Lem_G/H->G}
Let $G$ and $H$ be as above. If both $H$ and $G/H$ have Property (T)${}_c$, then so does $G$.
\end{lem}
\begin{proof}
Let $\pi\in\mathcal{R}_c(G)$ such that $\pi$ almost has invariant vectors. Hence the restriction $\pi_H$ to $H$ belongs to $\mathcal{R}_c(H)$, and it almost has invariant vectors. Let $\mathcal{H}_0\subset\mathcal{H}_\pi$ be the subspace of $\pi_H$-invariant vectors, which is non-trivial because $H$ has Property (T)${}_c$. Moreover, for all $\eta\in\mathcal{H}_0$, $t\in H$, $s\in G$,
\begin{align*}
\pi(t)\pi(s)\eta=\pi(s)\pi(s^{-1}ts)\eta=\pi(s)\eta,
\end{align*}
which shows that $\pi(s)\eta$ belongs to $\mathcal{H}_0$. We can thus restrict $\pi$ to a representation on $\mathcal{H}_0$. Furthermore, since every element of $\mathcal{H}_0$ is $H$-invariant, we get a representation $\pi_{G/H}:G/H\to\mathcal{B}(\mathcal{H}_0)$ defined by
\begin{align*}
\pi_{G/H}(q(s))\xi=\pi(s)\xi,\quad\forall s\in G,\ \forall \xi\in\mathcal{H}_0,
\end{align*}
where $q:G\to G/H$ is the quotient map. We will show that $\pi_{G/H}$ has a non-trivial invariant vector. Since $G/H$ has Property (T)${}_c$ and $\pi_{G/H}$ is a representation in $\mathcal{R}_c(G/H)$, it suffices to prove that $\pi_{G/H}$ almost has invariant vectors. Let $Q$ be a compact subset of $G/H$ and $\varepsilon>0$. There exists $\tilde{Q}\subset G$ compact such that $q(\tilde{Q})=Q$; see \cite[Lemma B.1.1]{BedlHVa}. Now let $p\in\tilde{A}_c(H)$ be the Kazhdan projection, and take $g\in C_{00}(H)$ such that $g\geq 0$, $\|g\|_1=1$, and
\begin{align*}
\|p-g\|_{\tilde{A}_c(H)}<\frac{1}{4}.
\end{align*}
Let $K=\operatorname{supp}(g)\subset G$. Since $\pi$ almost has invariant vectors, there is a unit vector $\xi\in\mathcal{H}_\pi$ such that
\begin{align*}
\|\pi(s)\xi-\xi\| < \min\left\{\frac{\varepsilon}{2c},\frac{1}{4}\right\},\quad\forall s\in\tilde{Q}\cup K.
\end{align*}
Let $(g_n)$ be a sequence in $C_{00}(H)$ such that $g_n\to p$. By Lemma \ref{Lem_tilde(g_n)}, for all $\eta\in\mathcal{H}_\pi$, $s\in G$,
\begin{align*}
\pi_H(p)\pi(s)\eta &=\lim_n \int_H g_n(t)\pi(t)\pi(s)\eta\, d\nu(t)\\
&=\pi(s)\left(\lim_n \int_H g_n(t)\pi(s^{-1}ts)\eta\, d\nu(t)\right)\\
&=\pi(s)\left(\lim_n \int_H \delta(s)^{-1}g_n(sts^{-1})\pi(t)\eta\, d\nu(t)\right)\\
&=\pi(s)\pi_H(p)\eta.
\end{align*}
Therefore
\begin{align*}
\|\pi(s)\pi_H(p)\xi-\pi_H(p)\xi\| &\leq \|\pi_H(p)\|\|\pi(s)\xi-\xi\|\\
&\leq c\frac{\varepsilon}{2c}\\
&= \frac{\varepsilon}{2},
\end{align*}
for all $s\in\tilde{Q}\cup K$. On the other hand,
\begin{align*}
\|\pi_H(p)\xi-\xi\| &\leq \|\pi_H(p)\xi-\pi_H(g)\xi\|+\|\pi_H(g)\xi-\xi\|\\
&\leq \|p-g\|_{\tilde{A}_c(H)} + \left\|\int_H g(s)(\pi(t)\xi-\xi) \,d\nu(t)\right\|\\
&\leq \frac{1}{4} + \int_K g(s)\|\pi(t)\xi-\xi\| \,d\nu(t)\\
&\leq \frac{1}{4} + \frac{1}{4}.
\end{align*}
This shows that $\|\pi_H(p)\xi\|\geq\frac{1}{2}$. Hence
\begin{align*}
\|\pi(s)\pi_H(p)\xi-\pi_H(p)\xi\| < \varepsilon\|\pi_H(p)\xi\|,\quad\forall s\in\tilde{Q}\cup K.
\end{align*}
Defining $\xi'=\pi_H(p)\xi\in\mathcal{H}_0$, we get
\begin{align*}
\|\pi_{G/H}(\dot{s})\xi'-\xi'\| < \varepsilon\|\xi'\|,\quad\forall \dot{s}\in Q.
\end{align*}
We conclude that $\pi_{G/H}$ almost has invariant vectors. By the discussion above, this implies that $\pi$ has a non-trivial invariant vector. Therefore $G$ has Property (T)${}_c$.
\end{proof}

\subsection{Lattices}\label{Subsec_latt}
A \textit{lattice} $\Gamma$ in a locally compact group $G$ is a closed discrete subgroup such that $G/\Gamma$ admits a regular, $G$-invariant probability measure $\mu$. We quickly review the definition of cocycles associated to lattices; we refer the reader to \cite[\S A.1]{Fur} for details. Choosing a Borel cross-section $\sigma:G/\Gamma\to G$, we can define a cocycle $b:G\times G/\Gamma\to\Gamma$ by
\begin{align*}
b(s,x)=\sigma(sx)^{-1}s\sigma(x),\quad\forall s\in G,\ \forall x\in G/\Gamma.
\end{align*}
This map satisfies the identity
\begin{align*}
b(st,x)=b(s,tx)b(t,x),\quad\forall s,t\in G,\ \forall x\in G/\Gamma.
\end{align*}

\begin{lem}\label{Lem_hat(pi)}
Let $G$ and $\Gamma$ be as above. Let $\pi\in\mathcal{R}_c(\Gamma)$ and let $\hat{\mathcal{H}}$ be the Hilbert space
\begin{align*}
\hat{\mathcal{H}}=L^2(G/\Gamma,\mu)\otimes\mathcal{H}_\pi \cong L^2(G/\Gamma; \mathcal{H}_\pi).
\end{align*}
Then the representation $\hat{\pi}:G\to\mathcal{B}(\hat{\mathcal{H}})$ given by
\begin{align*}
(\hat{\pi}(s)f)(x)=\pi(b(s,s^{-1}x))f(s^{-1}x),\quad\forall s\in G,\ \forall f\in\hat{\mathcal{H}},\ \forall x\in G/\Gamma,
\end{align*}
belongs to $\mathcal{R}_c(G)$.
\end{lem}
\begin{proof}
First observe that, for all $f\in\hat{\mathcal{H}}$, $s,t\in G$, $x\in G/\Gamma$,
\begin{align*}
(\hat{\pi}(s)\hat{\pi}(t)f)(x)&=\pi(b(s,s^{-1}x))(\hat{\pi}(t)f)(s^{-1}x)\\
&=\pi(b(s,s^{-1}x))\pi(b(t,t^{-1}s^{-1}x))f(t^{-1}s^{-1}x)\\
&=\pi(b(st,t^{-1}s^{-1}x))f(t^{-1}s^{-1}x)\\
&=(\hat{\pi}(st)f)(x).
\end{align*}
Moreover,
\begin{align*}
\|\hat{\pi}(s)f\|_{\hat{\mathcal{H}}}^2&=\int_{G/\Gamma} \left\|\pi(b(s,s^{-1}x))f(s^{-1}x)\right\|_{\mathcal{H}_\pi}^2\, d\mu(x)\\
&\leq c^2\int_{G/\Gamma} \left\|f(s^{-1}x)\right\|_{\mathcal{H}_\pi}^2\, d\mu(x)\\
&= c^2\|f\|_{\hat{\mathcal{H}}}^2.
\end{align*}
Here we used the fact that $\mu$ is $G$-invariant. This shows that $\hat{\pi}$ is a well-defined representation on $\hat{\mathcal{H}}$, with $|\hat{\pi}|\leq c$. It only remains to show that it is continuous. For this, let us consider an equivalent definition of the space $\hat{\mathcal{H}}$. Let $\Omega=\sigma(G/\Gamma)\subset G$ be the fundamental domain associated to the cross-section $\sigma$. The restriction of the Haar measure to $\Omega$ coincides with the push-forward measure $\sigma_*\mu$; see e.g. \cite[Corollary B.1.7]{BedlHVa}. We define $\mathcal{K}$ as the space of (equivalence classes of) measurable functions $g:G\to\mathcal{H}_\pi$ such that
\begin{align*}
g(st)=\pi(t)^{-1}g(s),\quad\forall s\in G,\ \forall t\in\Gamma,
\end{align*}
and
\begin{align*}
\|g\|_{\mathcal{K}}^2=\int_{\Omega}\|g(\omega)\|^2\, d\omega <\infty.
\end{align*}
Then there is an isometry $V:\hat{\mathcal{H}}\to\mathcal{K}$ given by
\begin{align}\label{V:H->K}
Vf(\omega t)=\pi(t)^{-1}f(q(\omega)),\quad\forall\omega\in\Omega,\ \forall t\in\Gamma,
\end{align}
where $q:G\to G/\Gamma$ denotes the quotient map. Observe that this is well defined since every $s\in G$ admits a unique decomposition $s=\omega t$ with $\omega\in\Omega$, $t\in\Gamma$. Moreover, it is an isometry because
\begin{align*}
\|Vf\|_{\mathcal{K}}^2=\int_\Omega \|f(q(\omega))\|_{\mathcal{H}_\pi}\, d(\sigma_*\mu)(\omega)
=\int_{G/\Gamma}\|f(x)\|_{\mathcal{H}_\pi}^2\,d\mu(x),
\end{align*}
for all $f\in\hat{\mathcal{H}}$. The inverse $V^{-1}$ is given by
\begin{align*}
V^{-1}g(x)=g(\sigma(x)),\quad\forall x\in G/\Gamma.
\end{align*}
Now define $\tilde{\pi}:G\to\mathcal{B}(\mathcal{K})$ by
\begin{align*}
\tilde{\pi}(s)=V\hat{\pi}(s)V^{-1},\quad\forall s\in G.
\end{align*}
For all $g\in\mathcal{K}$, $s\in G$, $\omega\in\Omega$, $t\in\Gamma$,
\begin{align}\label{VpiV-1}
\tilde{\pi}(s)g(\omega t) &= \pi(t)^{-1}\hat{\pi}(s)V^{-1}g(q(\omega))\notag\\
&= \pi(t^{-1}b(s,q(s^{-1}\omega)))g(\sigma(q(s^{-1}\omega)))\notag\\
&= g(\sigma(q(s^{-1}\omega))\sigma(q(s^{-1}\omega))^{-1}s^{-1}\sigma(q(\omega))t)\notag\\
&=g(s^{-1}\omega t).
\end{align}
With this, it is clear that the map
\begin{align*}
s\in G \mapsto \tilde{\pi}(s)g\in\mathcal{K}
\end{align*}
is continuous for every continuous element $g\in\mathcal{K}$. By density, we conclude that the same holds for every $g\in\mathcal{K}$. Hence $\tilde{\pi}$ belongs to $\mathcal{R}_c(G)$, and therefore so does $\hat{\pi}$.
\end{proof}

Let us define $\Psi:L^1(G)\to\ell^1(\Gamma)$ by
\begin{align}\label{Psi:L1}
\Psi(f)(t)=\int_G f(s)\mu(\{x\in G/\Gamma \, :\, b(s,x)=t\})\, ds,\quad\forall f\in L^1(G).
\end{align}
This map is a contraction. Indeed,
\begin{align*}
\|\Psi(f)\|_1 \leq\int_G |f(s)|\left(\sum_{t\in\Gamma}\mu(\{x\in G/\Gamma \, :\, b(s,x)=t\})\right)ds = \|f\|_1.
\end{align*}

\begin{lem}\label{Lem_Psi*=Phi}
Let $G$ and $\Gamma$ be as above. For all $c\geq 1$, the map \eqref{Psi:L1} extends to a contraction $\Psi:\tilde{A}_c(G)\to\tilde{A}_c(\Gamma)$, whose dual map $\Psi^*:B_c(\Gamma)\to B_c(G)$ is given by
\begin{align*}
\Psi^*(\varphi)(s)=\int_{G/\Gamma}\varphi(b(s,s^{-1}x))\, d\mu(x),\quad\forall\varphi\in B_c(\Gamma),\ \forall s\in G.
\end{align*}
\end{lem}
\begin{proof}
We will proceed backwards. For all $\varphi\in B_c(\Gamma)$, define
\begin{align*}
\Phi(\varphi)(s)=\int_{G/\Gamma}\varphi(b(s,s^{-1}x))\, d\mu(x),\quad\forall s\in G.
\end{align*}
We will show that this gives a contraction $\Phi:B_c(\Gamma)\to B_c(G)$, and that it is the dual map of $\Psi$. This will prove that $\Psi$ is a contraction. Let $\varphi\in B_c(\Gamma)$, and let $\pi\in\mathcal{R}_c(\Gamma)$ such that $\varphi=\langle\pi(\,\cdot\,)\xi,\eta\rangle$ for some $\xi,\eta\in\mathcal{H}_\pi$. Let $\hat{\pi}:G\to\mathcal{B}(\hat{\mathcal{H}})$ be the representation given by Lemma \ref{Lem_hat(pi)}. Now define $\hat{\xi}\in\hat{\mathcal{H}}$ as
\begin{align*}
\hat{\xi}=\mathds{1}\otimes\xi,
\end{align*}
where $\mathds{1}$ is the constant function $1$ on $G/\Gamma$. Define also $\hat{\eta}=\mathds{1}\otimes\eta$. Then
\begin{align*}
\langle\hat{\pi}(s)\hat{\xi},\hat{\eta}\rangle &=\int_{G/\Gamma}\langle\pi(b(s,s^{-1}x))\xi,\eta\rangle\, d\mu(x)\\
&=\int_{G/\Gamma}\varphi(b(s,s^{-1}x))\, d\mu(x)\\
&=\Phi(\varphi)(s),
\end{align*}
for all $s\in G$. This shows that $\Phi(\varphi)$ belongs to $B_c(G)$, and
\begin{align*}
\|\Phi(\varphi)\|_{B_c(G)}\leq \|\varphi\|_{B_c(\Gamma)}.
\end{align*}
Now, for every $f\in L^1(G)$ and $\varphi\in B_c(\Gamma)$,
\begin{align*}
\langle\Phi(\varphi),f\rangle &= \int_G\int_{G/\Gamma}\varphi(b(s,s^{-1}x))f(s)\, d\mu(x)\, ds\\
&= \int_G\int_{G/\Gamma}\varphi(b(s,x))f(s)\, d\mu(x)\, ds\\
&= \sum_{t\in\Gamma}\varphi(t)\int_G\mu(\{x\in G/\Gamma \, :\, b(s,x)=t\})f(s)\, ds\\
&=\langle\varphi,\Psi(f)\rangle.
\end{align*}
We conclude that $\Psi$ extends to a contraction $\tilde{A}_c(G)\to\tilde{A}_c(\Gamma)$, and that $\Phi=\Psi^*$.
\end{proof}

\begin{cor}\label{Cor_G->Gamma}
Let $G$ and $\Gamma$ be as above. If $G$ has Property (T)${}_c$, then so does $\Gamma$.
\end{cor}
\begin{proof}
Let $p\in\tilde{A}_c(G)$ be the invariant mean on $B_c(G)$. We claim that $\Psi(p)$ is the invariant mean on $B_c(\Gamma)$, where $\Psi$ is the map given by Lemma \ref{Lem_Psi*=Phi}. Indeed, the identity
\begin{align*}
\langle\varphi,\Psi(p)\rangle=\langle\Psi^*(\varphi),p\rangle,\quad\forall\varphi\in B_c(\Gamma),
\end{align*}
shows that $\Psi(p)$ is a mean. We need to prove that it is invariant. Let $\varphi\in B_c(\Gamma)$, $\pi\in\mathcal{R}_c(\Gamma)$, and $\xi,\eta\in\mathcal{H}_\pi$ such that $\varphi=\langle\pi(\,\cdot\,)\xi,\eta\rangle$. Let $t\in\Gamma$, and observe that
\begin{align*}
t\cdot\varphi(r)=\langle\pi(r)\xi,\pi(t^{-1})^*\eta\rangle,\quad\forall r\in\Gamma.
\end{align*}
Let $\hat{\pi}:G\to\mathcal{B}(\hat{\mathcal{H}})$ be the representation given by Lemma \ref{Lem_hat(pi)}. As in the proof Lemma \ref{Lem_Psi*=Phi}, we can write
\begin{align*}
\Psi^*(t\cdot\varphi)(s)=\langle\hat{\pi}(s)\hat{\xi},\hat{\eta}_t\rangle,\quad\forall s\in G,
\end{align*}
where $\hat{\xi}=\mathds{1}\otimes\xi$ and $\hat{\eta}_t=\mathds{1}\otimes\pi(t^{-1})^*\eta$. Thus
\begin{align*}
\langle\Psi^*(t\cdot\varphi),p\rangle=\langle\hat{\pi}(p)\hat{\xi},\hat{\eta}_t\rangle.
\end{align*}
Since $\hat{\pi}(p)$ is a projection onto the subspace of $\hat{\pi}$-invariant vectors, we have
\begin{align}\label{pi(sp)xi}
\hat{\pi}(s)\hat{\pi}(p)\hat{\xi}=\hat{\pi}(p)\hat{\xi},\quad\forall s\in G.
\end{align}
Define $V:\hat{\mathcal{H}}\to\mathcal{K}$ as in \eqref{V:H->K}. Then $V\hat{\pi}(p)\hat{\xi}$ is an invariant vector for the representation $\tilde{\pi}=V\hat{\pi}(\,\cdot\,)V^{-1}$ given by \eqref{VpiV-1}. Therefore $V\hat{\pi}(p)\hat{\xi}$ is a constant function from $G$ to $\mathcal{H_\pi}$. This implies that there exists $\xi_0\in\mathcal{H}_\pi$ such that
\begin{align*}
\hat{\pi}(p)\hat{\xi}(x)=\xi_0,\quad\forall x\in G/\Gamma.
\end{align*}
By \eqref{pi(sp)xi}, we have
\begin{align*}
\pi(b(s,s^{-1}x))\xi_0=\xi_0,\quad\forall s\in G,\ \forall x\in G/\Gamma.
\end{align*}
Equivalently,
\begin{align*}
\pi(t)\xi_0=\xi_0,\quad\forall t\in\Gamma.
\end{align*}
Here we used the fact that $b$ is surjective, which is true because, for every $t\in\Gamma$ and every $x\in G/\Gamma$, we can choose $s=\sigma(x)t\omega^{-1}$, where $\omega$ is any element of $\Omega=\sigma(G/\Gamma)$. Then
\begin{align*}
b(s,s^{-1}x)=t.
\end{align*}
Putting everything together,
\begin{align*}
\langle t\cdot\varphi,\Psi(p)\rangle &=\langle\Psi^*(t\cdot\varphi),p\rangle\\
&=\langle\hat{\pi}(p)\hat{\xi},\hat{\eta}_t\rangle\\
&=\int_{G/\Gamma}\langle\xi_0,\pi(t^{-1})^*\eta\rangle\, d\mu(x)\\
&=\int_{G/\Gamma}\langle\xi_0,\eta\rangle\, d\mu(x)\\
&=\langle\Psi^*(\varphi),p\rangle\\
&=\langle\varphi,\Psi(p)\rangle.
\end{align*}
We conclude that $\Psi(p)$ is the invariant mean on $B_c(\Gamma)$.
\end{proof}

The following lemma is essentially an adaptation of \cite[Theorem 1.7.1]{BedlHVa} to our setting.

\begin{lem}\label{Lem_Gamma->G}
Let $G$ and $\Gamma$ be as above. If $\Gamma$ has Property (T)${}_c$, then so does $G$.
\end{lem}
\begin{proof}
Let $\pi\in\mathcal{R}_c(G)$ such that $\pi$ almost has invariant vectors. Let us denote by $\pi_\Gamma$ the restriction of $\pi$ to $\Gamma$. Then $\pi_\Gamma$ belongs to $\mathcal{R}_c(\Gamma)$ and it almost has invariant vectors. Since $\Gamma$ has Property (T)${}_c$, $\pi_\Gamma$ has non-trivial invariant vectors. Let $p\in\tilde{A}_c(\Gamma)$ be the Kazhdan projection, and choose $g\in\C[\Gamma]$ such that $g\geq 0$, $\|g\|_1=1$ and
\begin{align*}
\|p-g\|_{\tilde{A}_c(\Gamma)}<\frac{1}{6(1+c)}.
\end{align*}
Since $\mu$ is regular, we can find $K\subset G/\Gamma$ compact such that
\begin{align*}
\mu(K)\geq 1-\frac{1}{7(1+c)}.
\end{align*}
Moreover, there exists $\tilde{K}\subset G$ compact such that $q(\tilde{K})=K$, where $q:G\to G/\Gamma$ is the quotient map; see \cite[Lemma B.1.1]{BedlHVa}. Now let $Q=\operatorname{supp}(g)$, and define $\tilde{Q}=(\tilde{K}\cup\{e\})Q$, where $e\in G$ is the identity element. Then $\tilde{Q}$ is compact, contains $Q$, and satisfies $q(\tilde{Q})\supset K$. Let $\xi\in\mathcal{H}_\pi$ be a unit vector such that
\begin{align*}
\|\pi(s)\xi-\xi\|<\frac{1}{6(1+c)},\quad\forall s\in\tilde{Q},
\end{align*}
which exists because $\pi$ almost has invariant vectors. Hence
\begin{align*}
\|\pi_\Gamma(p)\xi-\xi\| &\leq \|\pi_\Gamma(p)\xi-\pi_\Gamma(g)\xi\|+\|\pi_\Gamma(g)\xi-\xi\|\\
&\leq \|p-g\|_{\tilde{A}_c(\Gamma)} + \left\|\sum_{t\in Q}g(t)(\pi(t)\xi-\xi)\right\|\\
&\leq \frac{1}{6(1+c)} + \sum_{t\in Q}g(t)\|\pi(t)\xi-\xi\|\\
&\leq \frac{1}{3(1+c)}.
\end{align*}
Let $\xi'=\pi_\Gamma(p)\xi$, and observe that
\begin{align*}
\|\xi'\| &\leq \|\xi\|+\|\xi'-\xi\|\\
&\leq 1+\frac{1}{3(1+c)}\\
&\leq \frac{7}{6}.
\end{align*}
Now define
\begin{align*}
\eta=\int_{G/\Gamma}\pi(\sigma(x))\xi'\, d\mu(x),
\end{align*}
where $\sigma:G/\Gamma\to G$ is any Borel cross-section. Observe that, since $\xi'$ is $\pi_\Gamma$-invariant, the definition of $\eta$ does not depend on the choice of $\sigma$. Let $s\in G$. Using the fact that $\sigma(sx)^{-1}s\sigma(x)$ belongs to $\Gamma$, we see that
\begin{align*}
\pi(s)\eta&=\int_{G/\Gamma}\pi(s\sigma(x))\xi'\, d\mu(x)\\
&=\int_{G/\Gamma}\pi(\sigma(sx))\xi'\, d\mu(x)\\
&=\int_{G/\Gamma}\pi(\sigma(x))\xi'\, d\mu(x)\\
&=\eta.
\end{align*}
Thus $\eta$ is a $\pi$-invariant vector. It only remains to show that $\eta\neq 0$. First observe that, for every $s\in\tilde{Q}$,
\begin{align*}
\|\pi(s)\xi'-\xi'\| &\leq \|\pi(s)(\xi'-\xi)-(\xi'-\xi)\| + \|\pi(s)\xi-\xi\|\\
&\leq (1+c)\|\xi'-\xi\| + \frac{1}{6(1+c)}\\
&\leq\frac{1}{3}+\frac{1}{6(1+c)}\\
&\leq\frac{5}{12}.
\end{align*}
Hence
\begin{align*}
\|\eta-\xi'\| &= \left\|\int_{G/\Gamma}\pi(\sigma(x))\xi'-\xi' \, d\mu(x)\right\|\\
&\leq \int_{q(\tilde{Q})}\|\pi(\sigma(x))\xi'-\xi'\|\, d\mu(x) + (1-\mu(K))(1+c)\|\xi'\|\\
&\leq \frac{5}{12} + \frac{1}{7(1+c)}(1+c)\frac{7}{6}\\
&= \frac{7}{12}.
\end{align*}
Finally,
\begin{align*}
\|\eta\| &\geq \|\xi\| - \|\xi-\xi'\| - \|\eta-\xi'\|\\
&\geq 1 - \frac{1}{6} - \frac{7}{12}\\
&= \frac{1}{4}.
\end{align*}
Therefore $\eta$ is a non-trivial $\pi$-invariant vector.
\end{proof}

\section{{\bf The constant $c_{\operatorname{ub}}(G)$}}\label{Sec_c_ub}

In this section, we define the constant $c_{\operatorname{ub}}(G)$ and prove the remaining items of Theorem \ref{Thm_c_ub}. The following result is clear from the definition of Property (T)${}_c$. Thanks to Theorem \ref{Thm_equiv}, it can also be deduced from the weak*-weak*-continuity of the inclusion $B_{c_2}(G)\hookrightarrow B_{c_1}(G)$.

\begin{lem}
Let $G$ be a locally compact group and $c_1\geq c_2\geq 1$. If $G$ has Property (T)${}_{c_1}$, then it has Property (T)${}_{c_2}$.
\end{lem}
\begin{proof}
The identity $L^1(G)\to L^1(G)$ extends to a contraction $j:\tilde{A}_{c_1}(G)\to\tilde{A}_{c_2}(G)$ whose dual map $j^*:B_{c_2}(G)\to B_{c_1}(G)$ is the inclusion. If $p\in\tilde{A}_{c_1}(G)$ is the invariant mean on $B_{c_1}(G)$, then $j(p)\in\tilde{A}_{c_2}(G)$ is the invariant mean on $B_{c_2}(G)$.
\end{proof}

For every locally compact group $G$, we define $c_{\operatorname{ub}}(G)\in[1,\infty]$ as follows. If $G$ does not have Property (T), we set $c_{\operatorname{ub}}(G)=1$. If $G$ has Property (T),
\begin{align*}
c_{\operatorname{ub}}(G)=\sup\{c\geq 1 \ :\ G\ \text{ has Property (T)}_c \}.
\end{align*}

The following lemma is a consequence of \cite[Corollary 5.5]{dlS}.

\begin{lem}\label{Lem_(T)->c}
Let $G$ be a locally compact group. Then $c_{\operatorname{ub}}(G)>1$ if and only if $G$ has Property (T).
\end{lem}
\begin{proof}
If $c_{\operatorname{ub}}(G)>1$, then $G$ has Property (T) by definition. Conversely, if $G$ has Property (T), then \cite[Corollary 5.5]{dlS} ensures that there exists $\varepsilon>0$ such that the completion of $C_{00}(G)$ for the norm
\begin{align*}
\|f\|_{\mathcal{C}_{\mathcal{F}}}=\sup_{\pi\in\mathcal{F}}\|\pi(f)\|
\end{align*}
has a central Kazhdan projection, where $\mathcal{F}$ is a class of representations containing $\mathcal{R}_{1+\varepsilon}(G)$. Moreover, this projection is positive; see \cite[\S 3.5]{dlS}. Hence $G$ has Property (T)${}_{1+\varepsilon}$, and therefore $c_{\operatorname{ub}}(G)>1$. This can also be argued as follows. If $G$ has Property (T), then it is compactly generated; see \cite[Theorem 1.3.1]{BedlHVa}. In particular, it is $\sigma$-compact. Then, by \cite[Theorem 1.6]{FisMar}, it has Property (FH)${}_{1+\varepsilon}$ for some $\varepsilon>0$. By Proposition \ref{Prop_Guich_c}, it has Property (T)${}_{1+\varepsilon}$.
\end{proof}

Now we briefly discuss unitarisable groups; for a much more detailed presentation, we refer the reader to \cite{Pis2}. Let $\pi:G\to\mathcal{B}(\mathcal{H})$ be a uniformly bounded representation. We say that $\pi$ is \textit{unitarisable} if there is an invertible operator $S\in\mathcal{B}(\mathcal{H})$ such that the representation $S\pi(\,\cdot\,)S^{-1}$ is unitary. A group $G$ is said to be \textit{unitarisable} if every uniformly bounded representation of $G$ is unitarisable. Every amenable group is unitarisable, but is not known if the converse holds. Observe that, if $G$ is amenable, then $c_{\operatorname{ub}}(G)\in\{1,\infty\}$. In this case, $c_{\operatorname{ub}}(G)=\infty$ if and only if $G$ is compact.

\begin{lem}\label{Lem_unit->c}
Let $G$ be a unitarisable group. Then $c_{\operatorname{ub}}(G)\in\{1,\infty\}$.
\end{lem}
\begin{proof}
If $G$ does not have Property (T), then $c_{\operatorname{ub}}(G)=1$ by definition. Assume that $G$ has Property (T) and let $c\geq 1$. Let $\pi\in\mathcal{R}_c(G)$ such that $\pi$ almost has invariant vectors. Since $G$ is unitarisable, there is $S\in\mathcal{B}(\mathcal{H}_\pi)$ such that $\tilde{\pi}=S\pi(\,\cdot\,)S^{-1}$ is unitary. Moreover,
\begin{align*}
\|\tilde{\pi}(t)S\xi-S\xi\|\leq\|S\| \|\pi(t)\xi-\xi\|,\quad\forall t\in G,\ \forall\xi\in\mathcal{H}_\pi.
\end{align*}
This shows that $\tilde{\pi}$ almost has invariant vectors. Since $G$ has Property (T), $\tilde{\pi}$ has a non-trivial invariant vector $\xi_0$. Then $S^{-1}\xi_0$ is an invariant vector for $\pi$, which shows that $G$ has Property (T)${}_c$. Since $c$ was arbitrary, we conclude that $c_{\operatorname{ub}}(G)=\infty$.
\end{proof}

Now we are ready to prove Theorem \ref{Thm_c_ub}.

\begin{proof}[Proof of Theorem \ref{Thm_c_ub}]
Item (a) follows from Lemma \ref{Lem_(T)->c}. Item (c) is a consequence of Corollary \ref{Cor_G->G/H} and Lemma \ref{Lem_G/H->G}. Item (b) is a particular case of (c) for $G=G_1\times G_2$ and $H=G_i$ ($i=1,2$). Corollary \ref{Cor_G->Gamma} and Lemma \ref{Lem_Gamma->G} yield (d). Item (e) corresponds to Lemma \ref{Lem_unit->c}.
\end{proof}

\section{{\bf Von Neumann equivalence}}\label{Sec_vNE}

The aim of this section is to prove Theorem \ref{Thm_vNE}. We quickly review the definition of von Neumann equivalence; for a more detailed treatment, we refer the reader to \cite{IsPeRu}. Let $\Gamma$ and $\Lambda$ be two countable groups, and let $(\mathcal{M},\operatorname{Tr})$ be a von Neumann algebra endowed with a semi-finite, normal, faithful trace. We say that a trace preserving action $\Gamma\times\Lambda\curvearrowright^{\sigma}(\mathcal{M},\operatorname{Tr})$ is a von Neumann coupling between $\Gamma$ and $\Lambda$ if each of the individual actions $\Gamma,\Lambda\curvearrowright(\mathcal{M},\operatorname{Tr})$ admits a finite-trace fundamental domain. More precisely, there exist projections $q_\Gamma,q_\Lambda\in\mathcal{M}$ such that $\operatorname{Tr}(q_\Gamma),\operatorname{Tr}(q_\Lambda)<\infty$, and
\begin{align*}
\sum_{\gamma\in\Gamma}\sigma_\gamma(q_\Gamma)=\sum_{t\in\Lambda}\sigma_t(q_\Lambda)=1,
\end{align*}
where we view these sums as limits in the strong operator topology. We say that $\Gamma$ and $\Lambda$ are von Neumann equivalent ($\Gamma\sim_{\operatorname{vNE}}\Lambda$) if there is a von Neumann coupling between them.

Throughout this section, we will fix two countable groups $\Gamma, \Lambda$ and a von Neumnn coupling $\Gamma\times\Lambda\curvearrowright^{\sigma}(\mathcal{M},\operatorname{Tr})$. We will also normalise $\operatorname{Tr}$ so that $\operatorname{Tr}(q_\Lambda)=1$. There is a $\Lambda$-equivariant embedding $\theta_{q_\Lambda}:\ell^\infty(\Lambda)\to\mathcal{M}$ given by
\begin{align*}
\theta_{q_\Lambda}(\varphi)=\sum_{t\in\Lambda}\varphi(t)\sigma_t(q_\Lambda),\quad\forall\varphi\in\ell^\infty(\Lambda).
\end{align*}
Following \cite{Ish}, we consider the map $\Phi^*:\ell^\infty(\Lambda)\to\ell^\infty(\Gamma)$ given by
\begin{align}\label{map_Ish*}
\Phi^*(\varphi)(\gamma)=\operatorname{Tr}(\theta_{q_\Lambda}(\varphi)\sigma_{\gamma^{-1}}(q_\Lambda)),\quad\forall\varphi\in\ell^\infty(\Lambda),\ \forall\gamma\in\Gamma.
\end{align}
As the notation suggests, this is a dual map, where $\Phi:\ell^1(\Gamma)\to\ell^1(\Lambda)$ is defined as
\begin{align}\label{map_Ish}
\Phi(f)(t)=\sum_{\gamma\in\Gamma}\operatorname{Tr}(\sigma_t(q_\Lambda)\sigma_{\gamma^{-1}}(q_\Lambda))f(\gamma),\quad\forall f\in\ell^1(\Gamma),\ \forall t\in\Lambda;
\end{align}
see \cite[Lemma 3.3]{Bat}.

In order to show that Property (T) is stable under von Neumann equivalence, a method was developed in \cite[\S 6.2]{IsPeRu} to induce unitary representations from von Neumann couplings. We will relate this induction procedure to the map $\Phi^*$. Let $\pi:\Lambda\to\mathcal{U}(\mathcal{H})$ be a unitary representation. There is a dual Hilbert $\mathcal{M}$-module structure on $\mathcal{M}\overline{\otimes}\mathcal{H}$ such that
\begin{align*}
\langle a\otimes\xi,b\otimes\eta\rangle_{\mathcal{M}}=\langle\eta,\xi\rangle a^*b,\quad\forall\xi,\eta\in\mathcal{H},\ \forall a,b\in\mathcal{M}.
\end{align*}
The groups $\Lambda$ and $\Gamma$ act on $\mathcal{M}\overline{\otimes}\mathcal{H}$ by $\sigma\otimes\pi$ and $\sigma\otimes\operatorname{id}$ respectively. Let $(\mathcal{M}\overline{\otimes}\mathcal{H})^\Lambda$ denote the space of fixed points for the $\Lambda$-action, and let $\tau$ be the trace on $\mathcal{M}^\Lambda$ defined by
\begin{align*}
\tau(a)=\operatorname{Tr}(q_\Lambda a),\quad\forall a\in\mathcal{M}^\Lambda.
\end{align*}
Then there is an inner product on $(\mathcal{M}\overline{\otimes}\mathcal{H})^\Lambda$ given by
\begin{align*}
\langle y,x\rangle=\tau(\langle x,y\rangle_{\mathcal{M}}),\quad\forall x,y\in(\mathcal{M}\overline{\otimes}\mathcal{H})^\Lambda.
\end{align*}
Let $\tilde{\mathcal{H}}$ be the Hilbert space obtained by completing $(\mathcal{M}\overline{\otimes}\mathcal{H})^\Lambda$ for this inner product. Then the $\Gamma$-action extends to a unitary representation $\pi_{\mathcal{M}}:\Gamma\to\mathcal{U}(\tilde{\mathcal{H}})$.

Let $C^*(\Gamma)$ and $B(\Gamma)$ be the full C${}^*$-algebra and the Fourier--Stieltjes algebra of $\Gamma$. In our language, $C^*(\Gamma)=\tilde{A}_1(\Gamma)$ and $B(\Gamma)=B_1(\Gamma)$.

\begin{lem}\label{Lem_Ish_rep}
Let $\Gamma$ and $\Lambda$ be as above. Then the map \eqref{map_Ish} extends to a continuous map $\Phi:C^*(\Gamma)\to C^*(\Lambda)$ of norm $1$. Equivalently, the map \eqref{map_Ish*} restricts to a weak*-weak*-continuous map $\Phi^*:B(\Lambda)\to B(\Gamma)$ of norm 1. Moreover, if $\varphi$ is a coefficient of a representation $\pi:\Lambda\to\mathcal{U}(\mathcal{H})$, then $\Phi^*(\varphi)$ is a coefficient of the induced representation $\pi_{\mathcal{M}}$.
\end{lem}
\begin{proof}
Let $\pi:\Lambda\to\mathcal{U}(\mathcal{H})$ be a unitary representation and let $\xi,\eta\in\mathcal{H}$. Then the function $\varphi=\langle\pi(\,\cdot\,)\xi,\eta\rangle$ is an element of $B(\Lambda)$. Let $\pi_{\mathcal{M}}:\Gamma\to\mathcal{U}(\tilde{\mathcal{H}})$ be the induced representation, and consider the following elements of $\tilde{\mathcal{H}}$:
\begin{align*}
\tilde{\xi}&=\sum_{t\in\Lambda}\sigma_t(q_\Lambda)\otimes\pi(t)\xi, & \tilde{\eta}&=\sum_{t\in\Lambda}\sigma_t(q_\Lambda)\otimes\pi(t)\eta.
\end{align*}
These are clearly $\Lambda$-invariant vectors in $\mathcal{M}\overline{\otimes}\mathcal{H}$. Moreover,
\begin{align*}
\|\tilde{\xi}\|^2 = \sum_{s,t\in\Lambda}\langle\pi(s)\xi,\pi(t)\xi\rangle\operatorname{Tr}\left(q_\Lambda\sigma_t(q_\Lambda)\sigma_s(q_\Lambda)\right)
=\|\xi\|^2\operatorname{Tr}(q_\Lambda)=\|\xi\|^2.
\end{align*}
Similarly, $\|\tilde{\eta}\|=\|\eta\|$. Furthermore, for every $\gamma\in\Gamma$,
\begin{align*}
\langle\pi_{\mathcal{M}}(\gamma)\tilde{\xi},\tilde{\eta}\rangle
&= \sum_{s,t\in\Lambda}\langle\pi(s)\xi,\pi(t)\eta\rangle\operatorname{Tr}\left(q_\Lambda\sigma_t(q_\Lambda)\sigma_{\gamma}(\sigma_s(q_\Lambda))\right)\\
&= \sum_{s\in\Lambda}\varphi(s)\operatorname{Tr}\left(q_\Lambda\sigma_{\gamma}(\sigma_s(q_\Lambda))\right)\\
&=\Phi^*(\varphi)(\gamma).
\end{align*}
This shows that $\Phi^*:B(\Lambda)\to B(\Gamma)$ is a bounded linear map of norm at most $1$. By taking $\varphi=1$, we see that $\left\|\Phi^*\right\|=1$. The fact that $C^*(\Gamma)$ is the completion of $\ell^1(\Gamma)$ for the norm of $B(\Gamma)^*$ shows that $\Phi:C^*(\Gamma)\to C^*(\Lambda)$ is also bounded of norm $1$ and $\Phi^*$ is its dual map.
\end{proof}

The following is a refinement of the fact that Property (T) is a von Neumann equivalence invariant, which was proved in \cite[Theorem 1.2]{IsPeRu}.

\begin{lem}\label{Lem_Phi(P)}
Let $\Gamma$ and $\Lambda$ be as above, and assume that $\Gamma$ has Property (T). Then $\Lambda$ also has Property (T). Moreover, if $P$ denotes the Kazhdan projection of $C^*(\Gamma)$, then $\Phi(P)$ is the Kazhdan projection of $C^*(\Lambda)$.
\end{lem}
\begin{proof}
The fact that $\Lambda$ has Property (T) was proved in \cite[Theorem 1.2]{IsPeRu}. Let $\hat{\Lambda}$ denote the unitary dual of $\Lambda$, i.e. the space of (equivalence classes of) irreducible unitary representations of $\Lambda$. By \cite[Proposition 2]{Val2}, in order to prove that $\Phi(P)$ is the Kazhdan projection of $C^*(\Lambda)$, we only need to show that $1_\Lambda(\Phi(P))=1$, where $1_\Lambda:C^*(\Lambda)\to\C$ is the trivial representation, and that $\pi(\Phi(P))=0$ for every $\pi\in\hat{\Lambda}\setminus\{1_\Lambda\}$. First observe that the map $\Phi^*:B(\Lambda)\to B(\Gamma)$ from Lemma \ref{Lem_Ish_rep} is unital and positive. Since $P$ is a mean on $B(\Gamma)$, then $\Phi(P)$ is a mean on $B(\Lambda)$, which shows that
\begin{align*}
1_\Lambda(\Phi(P))=\langle 1,\Phi(P)\rangle_{B(\Lambda),C^*(\Lambda)}=1.
\end{align*}
Now let $\pi\in\hat{\Lambda}\setminus\{1_\Lambda\}$, and let $\xi,\eta$ be elements of the associated Hilbert space. Let $\varphi\in B(\Lambda)$ be given by
\begin{align*}
\varphi(s)=\langle\pi(s)\xi,\eta\rangle,\quad\forall s\in\Lambda.
\end{align*}
Then
\begin{align*}
\langle\pi(\Phi(P))\xi,\eta\rangle &= \langle\varphi,\Phi(P)\rangle_{B(\Lambda),C^*(\Lambda)}\\
&= \langle\Phi^*(\varphi),P\rangle_{B(\Gamma),C^*(\Gamma)}.
\end{align*}
On the other hand, by Lemma \ref{Lem_Ish_rep}, there are $\tilde{\xi},\tilde{\eta}\in\tilde{\mathcal{H}}$ such that
\begin{align*}
\Phi^*(\varphi)(\gamma)=\langle\pi_{\mathcal{M}}(\gamma)\tilde{\xi},\tilde{\eta}\rangle,\quad\forall\gamma\in\Gamma,
\end{align*}
where $\pi_{\mathcal{M}}:\Gamma\to\mathcal{U}(\tilde{\mathcal{H}})$ is the induced representation. This shows that
\begin{align*}
\langle\pi(\Phi(P))\xi,\eta\rangle = \langle\pi_{\mathcal{M}}(P)\tilde{\xi},\tilde{\eta}\rangle.
\end{align*}
Now, since $\pi$ is irreducible, it is weakly mixing, meaning that it does not contain finite-dimensional subrepresentations; see \cite[\S 8]{Fur2} for more details. By \cite[Proposition 6.15]{IsPeRu}, $\pi_{\mathcal{M}}$ does not have non-trivial invariant vectors, which means that $\pi_{\mathcal{M}}(P)=0$. We conclude that
\begin{align*}
\langle\pi(\Phi(P))\xi,\eta\rangle = 0.
\end{align*}
Since $\xi$ and $\eta$ were arbitrary, this shows that $\pi(\Phi(P))=0$, which is exactly what we wanted to prove.
\end{proof}

Recall that the Kazhdan projection of $C^*(\Gamma)$, when it exists, is the restriction to $B(\Gamma)$ of the unique invariant mean on $\operatorname{WAP}(\Gamma)$. Therefore, if we assume that $\Gamma$ has Property (T)${}_{c_1}$ and we want to prove that $\Lambda$ has Property (T)${}_{c_2}$ for some $c_1,c_2\geq 1$, by Lemma \ref{Lem_Phi(P)} it is sufficient to show that $\Phi^*$ restricts to a weak*-weak*-continuous map $B_{c_2}(\Lambda)\to B_{c_1}(\Gamma)$. It seems plausible that the argument in Lemma \ref{Lem_Ish_rep} can be adapted to uniformly bounded representations in order to show this holds for $c_1=c_2$; however, we shall follow a different strategy and prove this fact for $c_1>c_2$. This will be sufficient for proving Theorem \ref{Thm_vNE}.

We will need the following definition from \cite{Pis2}. Let $d\geq 1$ be an integer. The algebra $M_d(\Gamma)$ consists of all the functions $\varphi:\Gamma\to\C$ such that there are Hilbert spaces $\mathcal{H}_0,\ldots,\mathcal{H}_d$ with $\mathcal{H}_0=\mathcal{H}_d=\C$, and maps $\xi_i:\Gamma\to\mathcal{B}(\mathcal{H}_i,\mathcal{H}_{i-1})$ ($i=1,\ldots,d$) such that
\begin{align}\label{dec_M_d}
\varphi(s_1\cdots s_d)=\xi_1(s_1)\cdots\xi_d(s_d),\quad\forall s_1,\ldots,s_d\in\Gamma.
\end{align}
We endow it with the norm
\begin{align*}
\|\varphi\|_{M_d(\Gamma)}=\inf\left\{\sup_{s_1\in\Gamma}\|\xi_1(s_1)\|\cdots\sup_{s_d\in\Gamma}\|\xi_d(s_d)\|\right\},
\end{align*}
where the infimum is taken over all the decompositions as in \eqref{dec_M_d}. Observe that $M_1(\Gamma)=\ell^\infty(\Gamma)$. As shown in \cite[\S 2]{Pis2}, $B_c(\Gamma)$ is contained in $M_d(\Gamma)$ and
\begin{align}\label{M_d_leq_B_c}
\|\varphi\|_{M_d(\Gamma)}\leq c^d\|\varphi\|_{B_c(\Gamma)},
\end{align}
for all $\varphi\in B_c(\Gamma)$.

\begin{lem}\label{Lem_B_c+e}
Let $\Gamma$ and $\Lambda$ be as above. For every $c\geq 1$ and $\varepsilon>0$, $\Phi^*$ restricts to a weak*-weak*-continuous map $\Phi^*:B_c(\Lambda)\to B_{c+\varepsilon}(\Gamma)$ of norm at most $\frac{c+\varepsilon}{\varepsilon}$.
\end{lem}
\begin{proof}
As shown in \cite[Lemma 3.1]{Bat}, for every $d\geq 1$, $\Phi^*$ restricts to a bounded map $M_d(\Lambda)\to M_d(\Gamma)$ of norm $1$. Let $\varphi\in B_c(\Lambda)$. From \eqref{M_d_leq_B_c}, we see that
\begin{align*}
\|\Phi^*(\varphi)\|_{M_d(\Gamma)}\leq \|\varphi\|_{M_d(\Lambda)} \leq c^d\|\varphi\|_{B_c(\Lambda)},
\end{align*}
for all $d\geq 1$. On the other hand, by \cite[Theorem 2.6]{Pis2},
\begin{align*}
\|\Phi^*(\varphi)\|_{B_{c+\varepsilon}(\Gamma)} &\leq |\Phi^*(\varphi)(e)|+\sum_{d\geq 1}(c+\varepsilon)^{-d}\|\Phi^*(\varphi)\|_{M_d(\Gamma)}\\
&\leq |\varphi(e)|+\sum_{d\geq 1}\left(\frac{c}{c+\varepsilon}\right)^d\|\varphi\|_{B_c(\Lambda)}\\
&\leq \left(\frac{c+\varepsilon}{\varepsilon}\right)\|\varphi\|_{B_c(\Lambda)}.
\end{align*}
This shows that the map $\Phi^*:B_c(\Lambda)\to B_{c+\varepsilon}(\Gamma)$ has norm at most $\frac{c+\varepsilon}{\varepsilon}$. The fact that it is weak*-weak*-continuous follows from the definition of the norm of $\tilde{A}_c(\Lambda)$.
\end{proof}

\begin{proof}[Proof of Theorem \ref{Thm_vNE}]
Let $\Gamma$ and $\Lambda$ be countable groups such that $\Gamma\sim_{\operatorname{vNE}}\Lambda$, and assume that $\Gamma$ satisfies Property (T)${}_c$ for some $c\geq 1$. By Lemmas \ref{Lem_Phi(P)} and \ref{Lem_B_c+e}, $\Lambda$ satisfies Property (T)${}_{c-\varepsilon}$ for every $\varepsilon>0$. This shows that $c_{\operatorname{ub}}(\Lambda)\geq c_{\operatorname{ub}}(\Gamma)$. Reversing the roles of $\Gamma$ and $\Lambda$, we get $c_{\operatorname{ub}}(\Lambda)=c_{\operatorname{ub}}(\Gamma)$.
\end{proof}

\section{{\bf Rank 1 Lie groups}}\label{Sec_Lie}

In this section, we prove Theorem \ref{Thm_rank1}. Let $G$ be one the following rank 1 Lie groups: the isometry group of the $n$-dimensional quaternionic hyperbolic space $\operatorname{Sp}(n,1)$ ($n\geq 2$) or the exceptional group $F_{4,-20}$. In \cite{Doo}, Dooley constructed a family of uniformly bounded representations $(\pi_\zeta)$ of $G$ that approximates the trivial representation in such a way that $|\pi_\zeta|$ is uniformly bounded. Furthermore, this uniform bound is given by the Cowling--Haagerup constant of $G$; see \cite{CowHaa} for details on weak amenability and the Cowling--Haagerup constant.

We now review the elements of Dooley's construction that we will need in the proof of Theorem \ref{Thm_rank1}; we refer the reader to \cite{Doo} for a much more detailed presentation. Our main references for the structure of Lie groups and Lie algebras are the books \cite{Hel} and \cite{Kna}. For $G$ as above, let
\begin{align*}
\mathfrak{g}=\mathfrak{k}+\mathfrak{p}
\end{align*}
be the Cartan decomposition of its Lie algebra. Here $\mathfrak{k}$ and $\mathfrak{p}$ are the eigenspaces for the Cartan involution $\theta:\mathfrak{g}\to\mathfrak{g}$, associated to the eigenvalues $1$ and $-1$ respectively; see \cite[\S VI.2]{Kna} for details. We also have
\begin{align*}
\mathfrak{g}=\mathfrak{m}+\mathfrak{a}+\mathfrak{n}+\overline{\mathfrak{n}},
\end{align*}
where $\mathfrak{a}$ is a maximal abelian subalgebra of $\mathfrak{p}$, $\mathfrak{m}$ is the centraliser of $\mathfrak{a}$ in $\mathfrak{k}$, and
\begin{align*}
\mathfrak{n}&=\mathfrak{g}_\alpha+\mathfrak{g}_{2\alpha}, & \overline{\mathfrak{n}}&=\mathfrak{g}_{-\alpha}+\mathfrak{g}_{-2\alpha}.
\end{align*}
Here $\{-2\alpha,-\alpha,\alpha,2\alpha\}$ is the set of roots associated to $\mathfrak{a}$, and
\begin{align*}
\mathfrak{g}_\beta=\{X\in\mathfrak{g}\ \mid\ [H,X]=\beta(H)X\quad \forall H\in\mathfrak{a}\};
\end{align*}
see \cite[\S III.4]{Hel} for more details. At the Lie group level, this translates into the Bruhat big cell decomposition
\begin{align*}
\dot{G}=\overline{N}MAN,
\end{align*}
where $\dot{G}$ is a dense open submanifold of $G$, $A$ is a 1-dimensional abelian subgroup, and $N,\overline{N}$ are nilpotent subgroups. We denote by $r$ the homogeneous dimension:
\begin{align*}
r=\dim\mathfrak{g}_\alpha+2\dim\mathfrak{g}_{2\alpha}.
\end{align*}
The \textit{Weyl element} $w$ is a representative in $G$ of the non-trivial element of the Weyl group of $(\mathfrak{g},\mathfrak{a})$. It acts on $\mathfrak{a}$ via the adjoint representation $\operatorname{Ad}:G\to\operatorname{GL}(\mathfrak{g})$ by
\begin{align*}
\operatorname{Ad}(w)H=-H,\quad\forall H\in\mathfrak{a}.
\end{align*}
This implies that $\operatorname{Ad}(w)\mathfrak{n}=\overline{\mathfrak{n}}$, and therefore $wNw=\overline{N}$.

Let $C_{00}^\infty(\overline{N})$ be the space of smooth, compactly supported functions on $\overline{N}$. Then $G$ acts on $C_{00}^\infty(\overline{N})$ by linear transformations. In order to describe this action, we will identify $v\in\overline{N}$ with $(X,Y)\in\mathfrak{g}_\alpha \oplus \mathfrak{g}_{2\alpha}$. Fix $H_\alpha\in\mathfrak{a}$ such that $\alpha(H_\alpha)=1$, and define $a_t=\operatorname{exp}(tH_\alpha)$ for all $t\in\R$. For all $\xi\in\R$, we define $\pi_\xi$ as follows:
\begin{align}\label{pi_zeta}
\pi_\xi(a_t)f(X,Y)&=e^{t(\xi+r)/2}f(e^tX,e^{2t}Y),\quad\forall t\in\R,\notag\\
\pi_\xi(m)f(X,Y)&=f(\operatorname{Ad}(m)X,\operatorname{Ad}(m)Y),\quad\forall m\in M,\notag\\
\pi_\xi(u)f(v)&=f(vu),\quad\forall u\in\overline{N},\notag\\
\pi_\xi(w)f(v)&=f(v^\dagger),
\end{align}
where
\begin{align*}
(X,Y)^\dagger=\frac{1}{\mathcal{N}(X,Y)^4}\left(|X|^2X+[\theta X,Y],-Y\right),
\end{align*}
and
\begin{align*}
\mathcal{N}(X,Y)=(|X|^4+|Y|^2)^{1/4}.
\end{align*}
Let us mention that in \cite{Doo} $\xi$ is allowed to take values in $\C$; however, we will not need that level of generality. Now assume that $\dim\mathfrak{g}_\alpha+2<\xi<r$. For each $\eta>1$, we can define a scalar product on $C_{00}^\infty(\overline{N})$ by
\begin{align*}
\langle f,g\rangle_{\mathcal{H}_{\xi,\eta}} = \langle J_{\xi,\eta}\ast f,g\rangle_{L^2(\bar{N})},\quad\forall f,g\in C_{00}^\infty(\overline{N}).
\end{align*}
Here $J_{\xi,\eta}$ is defined as
\begin{align*}
J_{\xi,\eta}=(\eta\Omega_{\xi-r} -1)\mathcal{N}^{\xi-r},
\end{align*}
where $\Omega_{\xi-r}$ is a map with the following behaviour. If we write $\varepsilon=\frac{r-\xi}{2}$, then
\begin{align*}
\Omega_{\xi-r}(X,Y)=\phi(\varepsilon)\varphi_\varepsilon\left(\frac{|Y|}{\mathcal{N}(X,Y)^2}\right),\quad\forall X\in\mathfrak{g}_{\alpha},\ \forall Y\in\mathfrak{g}_{2\alpha},
\end{align*}
where $\varphi_\varepsilon$ is a $C^\infty$ function such that $\varphi_\varepsilon\to 1$ uniformly as $\varepsilon\to 0$, and $\phi(\varepsilon)$ converges to a number $\phi(0)$ as $\varepsilon\to 0$; see \cite[Proposition 3.6]{Doo}.\footnote{In \cite{Doo}, it is stated that $\phi(0)=0$. This seems to be a typo because this value is not consistent with the formula in \cite[Proposition 3.4]{Doo}.} We define $\mathcal{H}_{\xi,\eta}$ as the completion of $C_{00}^\infty(\overline{N})$ for this scalar product. The following was proved in \cite[Theorem 2.1]{Doo} and \cite[Theorem 3.1]{Doo}.

\begin{thm}[Dooley]\label{Thm_Doo}
Let $G=\operatorname{Sp}(n,1)$ ($n\geq 2$) or $G=F_{4,-20}$. For all $\eta>1$ and every $\xi\in(\dim\mathfrak{g}_\alpha+2,r)$, the action of $G$ on $C_{00}^\infty(\overline{N})$ described in \eqref{pi_zeta} extends to a uniformly bounded representation $\pi_{\xi,\eta}:G\to\mathcal{B}(\mathcal{H}_{\xi,\eta})$. Moreover,
\begin{align}\label{lim_eta}
\lim_{\eta\to 1}\lim_{\xi\to r}|\pi_{\xi,\eta}|=\boldsymbol\Lambda(G),
\end{align}
where $\boldsymbol\Lambda(G)$ is the Cowling--Haagerup constant \eqref{CHctt}.
\end{thm}

This theorem will allow us to show that $c_{\operatorname{ub}}(G)\leq\boldsymbol\Lambda(G)$ for these groups. We will need the following characterisation of the space $\operatorname{WAP}(G)$; see \cite[Theorem 1.4]{Vee}.

\begin{thm}[Veech]\label{Thm_Vee}
Let $G$ be a non-compact connected simple Lie group with finite centre. Then $\operatorname{WAP}(G)=C_0(G)\oplus\C 1$, and
\begin{align*}
m(\varphi)=\lim_{s\to\infty}\varphi(s),\quad\forall\varphi\in \operatorname{WAP}(G),
\end{align*}
where $m$ is the unique invariant mean on $\operatorname{WAP}(G)$.
\end{thm}

\begin{lem}\label{Lem_m(phi)=0}
Let $G$ and $\pi_{\xi,\eta}$ be as in Theorem \ref{Thm_Doo}, and let $m$ be the unique invariant mean on $\operatorname{WAP}(G)$. For every $f\in C_{00}^\infty(\overline{N})$, the function $\varphi:G\to\C$ defined by
\begin{align*}
\varphi(s)=\langle\pi_{\xi,\eta}(s)f,f\rangle,\quad\forall s\in G,
\end{align*}
satisfies $m(\varphi)=0$.
\end{lem}
\begin{proof}
By Lemma \ref{Lem_Bc_WAP}, $\varphi$ is indeed an element of $\operatorname{WAP}(G)$. By Theorem \ref{Thm_Vee},
\begin{align*}
m(\varphi)=\lim_{s\to\infty}\varphi(s).
\end{align*}
Recall that
\begin{align*}
\pi_{\xi,\eta}(u)f(s)=f(su),
\end{align*}
for all $u,s\in\overline{N}$. Moreover, since $f$ has compact support and $J_{\xi,\eta}$ tends to $0$ at infinity, $J_{\xi,\eta}\ast f$ belongs to $C_0(\overline{N})$. By the same argument, the function
\begin{align*}
u\in\overline{N}\mapsto\varphi(u)
\end{align*}
belongs to $C_0(\overline{N})$. Hence
\begin{align*}
m(\varphi)=\lim_{s\to\infty}\varphi(s)=\lim_{\substack{u\in\overline{N}\\ u\to\infty}}\varphi(u)=0.
\end{align*}
\end{proof}

Now we are ready to prove Theorem \ref{Thm_rank1}.

\begin{proof}[Proof of Theorem \ref{Thm_rank1}]
Let $(\pi_{\xi,\eta})_{\xi,\eta}$ be the family of representations given by Theorem \ref{Thm_Doo}. Fix $f_0\in C_{00}^\infty(\overline{N})$ such that $f_0\geq 0$ and $\|f_0\|_1=1$, and define
\begin{align*}
f_\eta=(\eta \phi(0)-1)^{-\frac{1}{2}} f_0 \in \mathcal{H}_{\xi,\eta}.
\end{align*}
Define also
\begin{align*}
g_\xi^{(s)}=\pi_{\xi,\eta}(s)f_0\in C_{00}^\infty(\overline{N}),
\end{align*}
for every $s\in G$ and $\xi\leq r$. Observe that this definition does not depend on $\eta$. Moreover $g_\xi^{(s)}$ depends on $\xi$ only for $s$ in $A$. Since $J_{\xi,\eta}\xrightarrow[\xi\to r]{}(\eta \phi(0)-1)$ almost everywhere, by the dominated convergence theorem,
\begin{align}\label{phi_e(s)}
\langle\pi_{\xi,\eta}(s)f_\eta,f_\eta\rangle_{\mathcal{H}_{\xi,\eta}} &= (\eta \phi(0)-1)^{-1}\langle J_{\xi,\eta}\ast g_\xi^{(s)},f_0\rangle_{L^2(\bar{N})}\notag\\
&\xrightarrow[\xi\to r]{} \langle 1\ast g_r^{(s)},f_0\rangle_{L^2(\bar{N})},
\end{align}
for all $s\in G$. In particular, for all $u\in\overline{N}$,
\begin{align*}
\langle\pi_{\xi,\eta}(u)f_\eta,f_\eta\rangle_{\mathcal{H}_{\xi,\eta}} 
&\xrightarrow[\xi\to r]{} \int_{\bar{N}}\int_{\bar{N}} f_0(s^{-1}tu)f_0(t)\, ds\, dt\\
&=\|f_0\|_1^2\\
&=1.
\end{align*}
Here we used the fact that $\overline{N}$ is unimodular, which is true because it is nilpotent. This also shows that
\begin{align*}
\lim_{\xi\to r}\|f_\eta\|_{\mathcal{H}_{\xi,\eta}}=1.
\end{align*}
Now let $c>\boldsymbol\Lambda(G)$. By \eqref{lim_eta}, we may take $\eta>1$ such that
\begin{align*}
\lim_{\xi\to r}|\pi_{\xi,\eta}|< c.
\end{align*}
For every $n\in\N$ large enough, define $\varphi_n\in B_{c}(G)$ by
\begin{align*}
\varphi_n(s)=\left\langle\pi_{r-\frac{1}{n},\eta}(s)f_\eta,f_\eta\right\rangle,\quad\forall s\in G.
\end{align*}
By Lemma \ref{Lem_m(phi)=0}, we have $m(\varphi_n)=0$ for all $n$. On the other hand, since $(\varphi_n)$ is a bounded sequence in $B_{c}(G)$, and $\tilde{A}_{c}(G)$ is separable, there is a subsequence $(\varphi_{n_k})$ and $\varphi\in B_{c}(G)$ such that
\begin{align*}
\varphi_{n_k}\to\varphi\quad\text{ in }\ \sigma(B_{c}(G),\tilde{A}_{c}(G)).
\end{align*}
Moreover, by \eqref{phi_e(s)}, we must have
\begin{align*}
\varphi(s)=\langle 1\ast g_r^{(s)},f_0\rangle_{L^2(\bar{N})},\quad\forall s\in G.
\end{align*}
In particular, for all $u\in\overline{N}$,
\begin{align*}
\varphi(u)=1.
\end{align*}
Thus
\begin{align*}
m(\varphi)=\lim_{s\to\infty}\varphi(s)=1.
\end{align*}
Therefore $m$ is not $\sigma(B_{c}(G),\tilde{A}_{c}(G))$-continuous. Equivalently, $G$ does not satisfy Property (T)${}_{c}$. Since this holds for every $c>\boldsymbol{\Lambda}(G)$, we conclude that $c_{\operatorname{ub}}(G)\leq \boldsymbol{\Lambda}(G)$. The inequality $c_{\operatorname{ub}}(G)>1$ follows from Lemma \ref{Lem_(T)->c} and the fact that $G$ has Property (T), which was proved in \cite{Kos}.
\end{proof}

Recall that a group $G$ has Property (PH)${}_c$ if it admits a proper affine action whose linear part is in $\mathcal{R}_c(G)$. As mentioned in the introduction, for simple Lie groups, there is a dichotomy between Property (FH)${}_c$ and Property (PH)${}_c$. This allows us to show that the groups $\operatorname{Sp}(n,1)$ and $F_{4,-20}$ satisfy Property (PH)${}_c$ for $c>\boldsymbol{\Lambda}(G)$.

\begin{proof}[Proof of Corollary \ref{Cor_rank1}]
Let $G$ be as above and $c>\boldsymbol{\Lambda}(G)$. By Theorem \ref{Thm_rank1}, $G$ does not have Property (T)${}_c$. Hence, by Proposition \ref{Prop_Guich_c}, it does not have Property (FH)${}_c$. This means that $G$ admits an unbounded cocycle $b$ for some $\pi\in\mathcal{R}_c(G)$. On the other hand, by \cite[Theorem 1.4]{Cor}, $G$ has Property $PL$, meaning that any uniformly Lipschitz action of $G$ on a metric space is either proper or it has bounded orbits; see \cite[Proposition 1.2]{Cor}. We conclude that the cocycle $b$ is necessarily proper. Therefore $G$ has Property (PH)${}_c$.
\end{proof}

\section{{\bf Random groups}}\label{Sec_random}
Now we focus on lower bounds for $c_{\operatorname{ub}}$ for random groups in the Gromov and triangular density models. In the Gromov model $\mathcal{G}(n,l,d)$, one fixes a density $d\in(0,1)$ and a number $n$. A group $\Gamma$ in $\mathcal{G}(n,l,d)$ is given by $\Gamma=\langle S|R\rangle$, where $|S|=n$ and $R$ is a set of relators, chosen uniformly among all subsets of cardinality $(2n-1)^{ld}$ of the set of cyclically reduced relators of length $l$. The idea is to study the behaviour when $l\to\infty$. The triangular model $\mathcal{M}(n,d)$ is defined in similar fashion, but in this case one considers only relators of length $3$, and studies the behaviour when the cardinality $n$ tends to infinity.

\.{Z}uk \cite{Zuk} proved that groups in the triangular model with $d > \frac{1}{3}$ have Property (T) with overwhelming probability (w.o.p.), meaning that
\begin{align*}
\lim_{n\to\infty}\P\left(\Gamma\text{ in }\mathcal{M}(n,d)\text{ has Property (T)}\right)=1.
\end{align*}
This result was extended to the Gromov model in \cite{KotKot}. More precisely, for $d > \frac{1}{3}$ and $n$ fixed,
\begin{align*}
\lim_{l\to\infty}\P\left(\Gamma\text{ in }\mathcal{G}(n,l,d)\text{ has Property (T)}\right)=1.
\end{align*}
Moreover, if $d<\frac{1}{2}$, these groups are infinite and hyperbolic w.o.p.; see \cite[\S 9.B]{Gro} for the Gromov model and \cite[Theorem 3]{Zuk} for the triangular model.

In \cite[Theorem 6.3]{Now}, Nowak took these ideas one step further and considered uniformly Lipschitz affine actions of random groups. In our language, his result can be stated as follows.

\begin{thm}[Nowak]\label{Thm_Now}
For every $c\in[1,\sqrt{2})$, groups in the Gromov density model with $\frac{1}{3} < d < \frac{1}{2}$ have Property (FH)${}_c$ w.o.p.
\end{thm}

By Proposition \ref{Prop_Guich_c}, the following is a direct consequence of Theorem \ref{Thm_Now}.

\begin{cor}
For every $c\in[1,\sqrt{2})$, groups in the Gromov density model with $\frac{1}{3} < d < \frac{1}{2}$ satisfy $c_{\operatorname{ub}}(\Gamma)\geq c$ w.o.p.
\end{cor}

Theorem \ref{Thm_Now} was later extended in \cite{Koi} to groups acting properly on certain 2-dimensional simplicial complexes, which again gives a lower bound for $c_{\operatorname{ub}}$ for such groups.

In \cite{dLadlS}, de Laat and de la Salle studied fixed point properties for actions of random groups on uniformly curved Banach spaces. This is a very rich class of Banach spaces that contains the class $\mathcal{E}_c$ defined in Section \ref{Sec_Guich}. To every uniformly curved Banach space $E$, one can associate a constant $\varepsilon(E)>0$ that provides a spectral criterion ensuring the existence of fixed points; see \cite[Theorem A]{dLadlS}. We will not go into much detail regarding this constant since we will only focus on spaces in $\mathcal{E}_c$, for which there is a very explicit lower bound. 

Recall that a group has Property $(F_E)$ if every isometric affine action on $E$ has a fixed point. The following was proved in \cite[Theorem C]{dLadlS}; see also \cite{Opp} for a different approach providing a strengthening of this result in the case of $L^p$ spaces.

\begin{thm}[de Laat--de la Salle]\label{Thm_dL-dlS}
Let $\eta\in(0,2)$. There is a constant $A>0$ and a sequence $(u_n)$ of positive real numbers tending to $0$ such that the following holds: For all $n\in\N$ and $d\in(0,1)$ satisfying
\begin{align*}
d\geq\frac{1}{3}+\frac{\log\log n - \log(2-\eta)}{3\log n},
\end{align*}
with probability $p\geq 1-u_n$, a group in $\mathcal{M}(n,d)$ has Property $(F_E)$ for every uniformly curved space $E$ with
\begin{align*}
\varepsilon(E)\geq\left(\frac{An}{(2n-1)^{3d}}\right)^\frac{1}{2}.
\end{align*}
\end{thm}

As mentioned above, we are only interested in the class $\mathcal{E}_c$. The following is a consequence of \cite[Remark 5.2]{dLadlS}.

\begin{prop}[de Laat--de la Salle]\label{Prop_dL-dlS}
There is a universal constant $K>0$ such that, for every $c\geq 1$ and every Banach space $E$ in $\mathcal{E}_c$, we have
\begin{align*}
\varepsilon(E)\geq\frac{K}{c^3}.
\end{align*}
\end{prop}

\begin{cor}\label{Cor_dL-dlS}
Let $\eta\in(0,2)$. There is a constant $K'>0$ and a sequence $(u_n)$ of positive real numbers tending to $0$ such that the following holds: For all $n\in\N$ and $d\in(0,1)$ satisfying
\begin{align*}
d\geq\frac{1}{3}+\frac{\log\log n - \log(2-\eta)}{3\log n},
\end{align*}
with probability $p\geq 1-u_n$, a group $\Gamma$ in $\mathcal{M}(n,d)$ satisfies
\begin{align*}
c_{\operatorname{ub}}(\Gamma)\geq K'\left(\frac{(2n-1)^{3d}}{n}\right)^\frac{1}{6}.
\end{align*}
\end{cor}
\begin{proof}
Fix $\eta\in(0,2)$ and let $A>0$, $(u_n)$ be the constant and the sequence given by Theorem \ref{Thm_dL-dlS}. Let $\Gamma$ be a group in $\mathcal{M}(n,d)$, and $c\geq 1$. We want to show that $\Gamma$ satisfies Property (FH)${}_c$ with probability $p\geq 1-u_n$. By Lemma \ref{Lem_corr_rep}, this is equivalent to showing that it has Property $(F_E)$ for every $E$ in $\mathcal{E}_c$. By Theorem \ref{Thm_dL-dlS}, this happens if
\begin{align*}
\varepsilon(E)\geq\left(\frac{An}{(2n-1)^{3d}}\right)^\frac{1}{2}.
\end{align*}
In light of Proposition \ref{Prop_dL-dlS}, it is sufficient to have
\begin{align*}
\frac{K}{c^3}\geq\left(\frac{An}{(2n-1)^{3d}}\right)^\frac{1}{2}.
\end{align*}
Defining $K'=\left(\frac{K}{\sqrt{A}}\right)^\frac{1}{3}$, this is equivalent to
\begin{align*}
c\leq K'\left(\frac{(2n-1)^{3d}}{n}\right)^\frac{1}{6}.
\end{align*}
By Proposition \ref{Prop_Guich_c}, with probability $p\geq 1-u_n$, $\Gamma$ has Property (T)${}_c$ for every $c$ satisfying the inequality above. This shows that, with probability $p$,
\begin{align*}
c_{\operatorname{ub}}(\Gamma)\geq K'\left(\frac{(2n-1)^{3d}}{n}\right)^\frac{1}{6}.
\end{align*}
\end{proof}

This result allows us to see that the implications \eqref{impl_T_c} are strict.

\begin{cor}\label{Cor_str_imp}
The following hold:
\begin{itemize}
\item[(i)] For every $c>1$, there is a group without Property (T${}^*$) that satisfies Property (T)${}_c$.
\item[(ii)] For every $c>3$, there is a group without Property (T)${}_c$ that satisfies Property (T). 
\end{itemize}
\end{cor}
\begin{proof}
By Corollary \ref{Cor_dL-dlS}, for every $c>1$, there is an infinite hyperbolic group $\Gamma$ with $c_{\operatorname{ub}}(\Gamma)>c$. On the other hand, hyperbolic groups are weakly amenable; see \cite{Oza2}. In particular, they satisfy the approximation property AP; see \cite[Theorem 1.12]{HaaKra}. Hence, by \cite[Proposition 5.5]{HaKndL}, they do not satisfy Property (T${}^*$). This proves the first statement. For the second one, simply consider the group $G=\operatorname{Sp}(2,1)$. By Theorem \ref{Thm_rank1}, we know that $1<c_{\operatorname{ub}}(G)\leq 3$, which implies the second statement.
\end{proof}

\section{{\bf Questions}}\label{Sec_ques}

We end this paper with a few questions related to Property (T)${}_c$ that we believe are interesting.

\subsection*{Question 1} Does every hyperbolic group $\Gamma$ satisfy $c_{\operatorname{ub}}(\Gamma)<\infty$?

By Theorems \ref{Thm_rank1} and \ref{Thm_c_ub}, we know that this is true for uniform lattices in $\operatorname{Sp}(n,1)$ and $F_{4,-20}$. Moreover, Lafforgue \cite{Laf} proved that hyperbolic groups do not have Strong Property (T). This was done by constructing a representation $\pi$ such that $s\mapsto\|\pi(s)\|$ has polynomial growth, and such that $\pi$ does not have non-trivial invariant vectors, but $s\mapsto\pi(s^{-1})^*$ does. This allows one to conclude that the Kazhdan projection cannot exist in this context. However, since this representation is not uniformly bounded, this argument does not apply to $\tilde{A}_c(\Gamma)$.

Another interesting remark is the fact that hyperbolic groups do not have Property (T${}^*$); see the proof of Corollary \ref{Cor_str_imp}. This raises another question: is Property (T${}^*$) equivalent to $c_{\operatorname{ub}}(\Gamma)=\infty$?

We also mention a related question attributed to Shalom (see \cite[Conjecture 35]{Now2} and \cite[Problem 14]{ObeR}) asking whether every hyperbolic group admits a proper uniformly Lipschitz affine action on a Hilbert space. Again, by Corollary \ref{Cor_rank1}, we know that this is true for uniform lattices in $\operatorname{Sp}(n,1)$ and $F_{4,-20}$, and it is a consequence of the fact that $c_{\operatorname{ub}}(\Gamma)<\infty$ for such groups.

\subsection*{Question 2} Can we use the family of representations constructed in \cite{CosMar} to show that $c_{\operatorname{ub}}(\Gamma)<\infty$, where $\Gamma=\operatorname{Mod}(\Sigma)$ is the mapping class group of a punctured surface?

Costantino and Martelli \cite{CosMar} constructed an analytic family of uniformly bounded representations $(\pi_z)_{z\in\D}$ of a mapping class group, defined by perturbations of the multicurve representation. This construction is inspired by the cocycle technique developed by Valette \cite{Val} for groups acting properly on trees, which provides a way of showing that these groups do not have Property (T) by a very similar argument as that of Theorem \ref{Thm_rank1}. It would be interesting to extend this argument to mapping class groups, perhaps for $c>1$ depending on the surface.

\subsection*{Question 3} Does Property (T)${}_c$ imply Property (FH)${}_c$?

Recall that the Delorme--Guichardet theorem says that, for a very large class of groups, Property (T) and Property (FH) are equivalent. In the context of the present work, it is natural to ask whether this can be extended to Properties (T)${}_c$ and (FH)${}_c$ for $c>1$. Moreover, we know from Proposition \ref{Prop_Guich_c} that one of the implication holds. However, the proof of Delorme's theorem uses positive definite functions in a crucial way; see \cite[Theorem 2.12.4]{BedlHVa}. These tools are no longer available in our more general context, and it is not clear how to adapt the arguments in the proof.

Let us mention that this question has been around for quite some time; see e.g. \cite[Problem 15]{ObeR}.

\bibliographystyle{plain} 

\bibliography{Bibliography}

\end{document}